\documentclass[12pt,a4paper]{article}
\usepackage[titletoc,title]{appendix}
\usepackage{amsmath,amsfonts,amssymb,bm,amsthm}
\usepackage{tikz}
\usetikzlibrary{shapes,arrows}
\usetikzlibrary{intersections,shapes.arrows}
\numberwithin{equation}{section}

\usepackage{ebgaramond}
\usepackage[T1]{fontenc}
\newtheoremstyle{plainsc}
{\topsep}
{\topsep}
{\itshape}
{}
{\large\scshape}
{}
{5pt plus 1pt minus 1pt}
{\thmname{#1}\thmnumber{ #2}\thmnote{ (#3)}.}

\newtheoremstyle{definitionsc}
{}
{}
{\normalfont}
{}
{\large\scshape}
{ }
{ }
{\thmname{#1}\thmnumber{ #2}\thmnote{ (#3)}.}

\newtheoremstyle{remarksc}
{0.5\topsep}
{0.5\topsep}
{\normalfont}
{}
{\large\itshape}
{}
{ }
{\thmname{#1}\thmnumber{ #2}\thmnote{ (#3)}.}

\theoremstyle{plainsc}
\newtheorem{theorem}{Theorem}[section]
\newtheorem{lemma}[theorem]{Lemma}

\theoremstyle{definitionsc}
\newtheorem{example}[theorem]{Example}
\newtheorem{definition}[theorem]{Definition}
\theoremstyle{remarksc}
\newtheorem{remark}[theorem]{Remark}

\usepackage[top=1in, bottom=1in, left=1in, right=1in]{geometry}
\usepackage{fancyhdr}
\newcommand{\dbar}{{\mkern3mu\mathchar'26\mkern-12mu \mathrm{d}}}
\newcommand{\ir}{\mathrm{i}}
\newcommand{\er}{\mathrm{e}}
\newcommand{\dr}{\mathrm{d}}
\usepackage{enumerate}
\usepackage{xcolor,cancel}

\usepackage[colorlinks,citecolor=blue,linkcolor=blue]{hyperref}

\begin{document}

\title{Geometric wave propagator on Riemannian manifolds}
\author{Matteo Capoferri\thanks{MC:
Department of Mathematics,
University College London,
Gower Street,
London WC1E~6BT,
UK;
(current address) 
Maxwell Institute for Mathematical Sciences, Edinburgh \emph{and}
Department of Mathematics, Heriot-Watt University,
Edinburgh EH14 4AS, UK;
m.capoferri@hw.ac.uk,
\url{https://mcapoferri.com}
}
\and
Michael Levitin\thanks{ML:
Department of Mathematics and Statistics,
University of Reading,
Pepper Lane,
Whiteknights,
Reading RG6~6AX,
UK;
M.Levitin@reading.ac.uk,
\url{http://www.michaellevitin.net}
}
\and
Dmitri Vassiliev\thanks{DV:
Department of Mathematics,
University College London,
Gower Street,
London WC1E~6BT,
UK;
D.Vassiliev@ucl.ac.uk,
\url{http://www.ucl.ac.uk/\~ucahdva/}
}}

\renewcommand\footnotemark{}

\date{17 July 2023}

\maketitle
\begin{abstract}
We study the propagator of the wave equation on a closed Riemannian manifold $M$. We propose a geometric approach to the construction of the propagator as a single oscillatory integral global both in space and in time with a distinguished complex-valued phase function.
This enables us to provide a global invariant definition of the full symbol of the propagator --- a scalar function on the cotangent bundle --- and an algorithm for the explicit calculation of its homogeneous components.
The central part of the paper is devoted to the detailed analysis of the subprincipal symbol; in particular, we derive its explicit small time asymptotic expansion. We present a general geometric construction that allows one to visua\-lise obstructions due to caustics and describe their circumvention with the use of a complex-valued phase function.
We illustrate the general framework with explicit examples in dimension two.

\


\

{\bf MSC classes: }
primary 35L05; secondary 58J40, 58J45, 35P20.

\end{abstract}

\tableofcontents

\section{Statement of the problem}

Let $(M,g)$ be a connected closed Riemannian manifold of dimension $d\ge2$.
We denote local coordinates on $M$ by $x^\alpha$, $\alpha=1,\ldots,d$.
The $L^2$ inner product on complex-valued functions is defined as
\begin{equation*}
\label{inner product}
(u,v):=\int_M\overline{u(x)}\,v(x)\,\rho(x)\,\dr x\,,
\end{equation*}
where
\begin{equation}\label{riemannian density}
\rho(x):=\sqrt{\det g_{\mu \nu}(x)}
\end{equation}
and $\dr x=\dr x^1\dots \dr x^d\,$.
The Laplace--Beltrami operator on scalar functions is
\begin{equation}
\label{laplacian on functions}
\Delta= \rho(x)^{-1} \dfrac{\partial}{\partial x^\mu} \,\rho(x)\,g^{\mu\nu}(x)\,\dfrac{\partial}{\partial x^\nu}\,.
\end{equation}
Here and further on we adopt Einstein's summation convention over repeated indices.
 
It is well known that the operator \eqref{laplacian on functions} is non-positive
and has discrete spectrum accumulating to~$-\infty$. We adopt the following notation
for the eigenvalues and normalised eigenfunctions of $-\Delta$,
\begin{equation*}
-\Delta v_k=\lambda_k^2 v_k\,,
\end{equation*}
where eigenvalues are enumerated with account of their multiplicity as
\[
0=\lambda_0<\lambda_1\le\lambda_2\le\ldots\le\lambda_k\le\ldots\to+\infty.
\]

Consider the Cauchy problem for the wave equation
\begin{subequations}
\begin{equation}
\label{wave equation}
\left(
\frac{
\partial^2}{\partial t^2}
-
\Delta
\right)
f(t,x)=0\,,
\end{equation}
\begin{equation}
\label{wave equation initial condition}
f(0,x)=f_0(x),
\qquad
\frac{
\partial f}{\partial t}(0,x)=f_1(x).
\end{equation}
\end{subequations}
Functional calculus allows one to write the solution of
\eqref{wave equation},
\eqref{wave equation initial condition}
as
\begin{equation}
\label{wave equation solution of Cauchy problem}
f=
\cos\bigl(t\,\sqrt{-\Delta}\,\bigr)
\,f_0
\,+\,
\sin\bigl(t\,\sqrt{-\Delta}\,\bigr)
\,(-\Delta)^{-1/2}
\,f_1
\,+\,
t\,(v_0\,,f_1)\,,
\end{equation}
where
\begin{equation*}
\label{pseudoinverse}
(-\Delta)^{-1/2}
:=
\sum_{k=1}^\infty
\frac{1}{\lambda_k}\,(v_k\,,\,\cdot\,)\,v_k
\end{equation*}
is the pseudoinverse of the operator $\sqrt{-\Delta}\,$
\cite[Chapter 2 Section 2]{rellich}.

The RHS of
\eqref{wave equation solution of Cauchy problem}
contains three operators:
$\cos\bigl(t\,\sqrt{-\Delta}\,\bigr)$,
$\sin\bigl(t\,\sqrt{-\Delta}\,\bigr)$
and
$(-\Delta)^{-1/2}$.
The first two are Fourier Integral Operators (FIOs), whereas the third one is a
pseudodifferential operator.
Assuming one has a good description of the operator $(-\Delta)^{-1/2}\,$
--- for which there is a well developed theory, see e.g.~\cite{Hor} ---
solving the Cauchy problem
\eqref{wave equation},
\eqref{wave equation initial condition}
reduces to constructing the FIO
\begin{equation}
\label{definition of propagator}
U(t)
:=
\er^{-\ir t\sqrt{-\Delta}}
=
\int u(t,x,y)\,(\,\cdot\,)\,\rho(y)\,\dr y\,,
\end{equation}
whose Schwartz kernel reads
\begin{equation}
\label{propagator}
u(t,x,y):=\sum_{k=0}^\infty \er^{-\ir t\lambda_k}\,v_k(x)\,\overline{v_k(y)}\,.
\end{equation}
The operator $U(t)$ is called the \emph{wave propagator} (of the Laplacian) and is the
(distributional) solution of
\begin{subequations}
\begin{equation}
\label{main PDE}
\left(
-\ir\,\frac{\partial}{\partial t}
+
\sqrt{-\Delta}
\right)
U(t)=0\,,
\end{equation}
\begin{equation}
\label{half wave equation initial condition}
U(0)=\operatorname{Id}.
\end{equation}
\end{subequations}

The goal of this paper is to provide
an explicit formula for the operator $U(t)$ modulo an integral operator with infinitely smooth integral kernel,
written as a single invariantly defined oscillatory integral global in space and in time.

The study of solutions of hyperbolic partial differential equations on manifolds --- and of the wave propagator in particular --- is a well established subject, both within and outside microlocal analysis. As far as microlocal methods are concerned, rigorous descriptions of the singular structure of the propagator, as well as   the construction of parametrices, can be found, for example, in \cite{Had}, \cite{Ri49,Ri60}, \cite{DuHo}, \cite[Vol.~3 \& 4]{Hor}, \cite{Tr}, \cite{Shu}. These publications rely on spectral-theoretic techniques, often combined with tools from the theory of local oscillatory integrals.

In this paper, we adopt a somewhat different \emph{global} approach, which originates from the works of Laptev, Safarov and Vassiliev \cite{LSV} and Safarov and Vassiliev \cite{SaVa}. They showed that it is possible to write the propagator for a fairly wide class of hyperbolic equations as \emph{one} single Fourier integral operator, global both in space and in time, provided one uses a \emph{complex-valued} phase function. This idea is not entirely new.
For instance, constructions which look very similar at a formal level, albeit lacking mathematical rigour, have been for a long time appearing in solid state physics papers on electromagnetic wave propagation, obviously inspired by geometric optics. The mathematical formalisation of these ideas often appears under the name of `Gaussian beams', see, e.g., \cite{ralston}.
In the realm of pure mathematics, FIOs with complex phase functions were considered, for example, by Melin and Sj\"ostrand \cite{MeSj}. The fundamental difference between their approach and the one presented here lies in the fact that not only they have complex-valued phase functions, but, unlike \cite{LSV}, \cite{SaVa}, they also work in a complexified phase space, which makes the analysis quite dissimilar.

Melin and Sj\"ostrand's techniques were later adopted by Zelditch in the construction of the wave group on real analytic manifolds, see, e.g., \cite{Zel1} and \cite{Zel2}. In his works, focussed on the study of nodal domains and nodal lines of complex eigenfunctions, the wave group appears as the composition of three Fourier integral operators. The general idea of his construction --- up to technical details --- goes as follows. Consider the complexification $M_\mathbb{C}$ of $M$ and let
\[
M_\tau:=\{\zeta\in M_\mathbb{C} \,|\, \sqrt{\mathfrak{r}}(\zeta)\le \tau \}
\]
be the Grauert tube of radius $\tau$ of $M$ within $M_\mathbb{C}$, $\sqrt{\mathfrak{r}}$ denoting the Grauert tube function. Furthermore, let 
\[
\partial M_\tau:=\{\zeta\in M_\mathbb{C}\,|\,\sqrt{\mathfrak{r}}(\zeta)=\tau\}.
\]
Then the wave propagator $e^{-\ir t \sqrt{-\Delta}}:L^2(M)\to L^2(M)$ is given by the composition of
\begin{itemize}
\item[(i)] an operator $P^\tau: L^2(M) \to \mathcal{O}^{\frac{d-1}{4}}(\partial M_\tau)\subset L^2(\partial M_\tau)$, the analytic extension of the Poisson semigroup $\er^{\tau \sqrt{-\Delta}}$;
\item[(ii)] an operator $T_{\Phi^t}$ on $\mathcal{O}^{\frac{d-1}{4}}(\partial M_\tau)$,
\[
T_{\Phi^t} f:= f\circ \Phi^t,
\]
 realising the translation along the geodesic flow $\Phi^t$; 
\item[(iii)] the adjoint of $P^\tau$, $(P^\tau)^*:\mathcal{O}^{\frac{d-1}{4}}(\partial M_\tau)\to L^2(M)$. 
\end{itemize}
One needs, additionally, to incorporate a pseudodifferential operator $S_t$ (multiplication by a symbol) in order to obtain, in the end, a unitary operator
\[
(P^\tau)^* \circ S_t \circ T_{\Phi^t} \circ P^\tau: L^2(M) \to L^2(M).
\]
Zelditch's approach consists, effectively, in writing the wave group $U(t)$ as the conjugation of the translation operator $T_{\Phi^t}$
by the (analytic extension of the) Poisson semigroup $P^\tau$. For further details on the operator $P^\tau$ we refer the reader to \cite{boutet}, \cite{Zel3}, \cite{lebeau} and \cite{stenzel}.
Despite some similarities in the idea of adopting a complex phase to achieve a representation global in time, our construction is overall quite different from Zelditch's one, as it will be clear later on.

The techniques from \cite{LSV}, \cite{SaVa} are rather abstract and do not account for any underlying geometry. This may be a reason why they have not been picked up by the wider mathematical community. 
There are only few subsequent publications using these methods as a fundamental tool. Laptev and Sigal \cite{LaSi} constructed the propagator for the magnetic Schr\"{o}\-din\-ger operator in flat Euclidean space for phase functions with purely quadratic imaginary part. Jakobson, Safarov and Strohmaier, when studying branching billiards on Riemannian manifolds with discontinuous metric in \cite{jakobson}, rely in their proofs on boundary layer oscillatory integrals with complex-valued phase function, in the spirit of \cite{SaVa}. Furthermore, Safarov set his programme on global calculi on manifolds \cite{Sa14, McSa} in the framework of \cite{LSV}.  An extension of results from \cite{SaVa} to first order systems of PDEs has been carried out by Chervova, Downes and Vassiliev \cite{CDV} in the process of computing two-term spectral asymptotics.

Laptev and Sigal's results mentioned above were improved and extended by Robert in \cite{robert}, where he constructs explicitly the Schwartz kernel of the quantum propagator for the Schr\"odinger operator on $\mathbb{R}^d$ as a Fourier integral operator with quadratic complex-valued phase function and semiclassical subquadratic symbol. Robert adopts a distinguished phase function adapted to the Hamiltonian dynamics, which, however, does not coincide with a specialisation to the flat case of the Levi-Civita phase function used in the current paper.

The construction of \cite{LSV,SaVa} works, strictly speaking, for closed manifolds or compact manifolds with boundary. The compactness assumption, however, is not essential and can be removed with some effort. Results in this direction, although in a different setting and without the use of complex-valued phase functions, have been recently obtained by Coriasco and collaborators \cite{CoSh, CDS}. In the current paper, we will refrain from carrying out such an extension and we will stick to the case of closed manifolds.

The general properties and the singular structure of the integral kernel $u$ of the wave propagator, see \eqref{propagator}, are well understood. At the same time very little is known when it comes to explicit formulae. In particular, almost no information on the symbol of $U(t)$ can be found in the literature. With the exception of those cases where all eigenvalues and eigenfunctions are known, the only general result available to date is that the principal symbol is $1$. In fact, we are unaware of any invariant definition of full symbol  --- or subprincipal symbol ---  for Fourier integral operators of the form \eqref{definition of propagator}. The goal of the current paper is to build upon \cite{LSV}, developing their construction further for the case of Riemannian manifolds. The geometric nature of our construction will allow us to provide invariant definitions of full and $g$-subprincipal symbol of the wave propagator, analyse them, and give explicit formulae. Here `$g$' is a reference to the Riemannian metric used in the construction of the phase function $\varphi$, leading up to the definition of the full symbol.

Our construction, although non-trivial, is quite natural and fully geometric in its building blocks. Among other things, we aim to show the potential of the method, which, due to the fact of being fully explicit, may find applications in pure and applied mathematics, as well as in other applied sciences. With this in mind, we will not pursue the standard microlocal approach involving half-densities, but, rather, we will adjust our theory to the case of operators acting on scalar functions.

One of the applications of our construction of the wave propagator is the
calculation of higher Weyl coefficients,
see Appendix~\ref{Weyl coefficients}.
For the Laplacian this can be done using a variety of alternative methods,
the simplest being the heat kernel and the resolvent approaches.
However, if one replaces the Laplace--Beltrami operator by a first order system,
whose spectrum is, in general, not semi-bounded, the heat kernel method can no longer
be applied, at least in its original form. Furthermore, even resolvent techniques
require major modification \cite{ASV}.
In the future we plan to apply our approach to first order systems of partial differential equations on Riemannian manifolds for which we expect to compute additional (compared to what is known in the current literature) Weyl coefficients.

The paper is structured as follows.

In Section~\ref{Lagrangian manifolds and Hamiltonian flows} we present a brief overview of the theory of global Lagrangian distributions and their relation to hyperbolic problems, as developed in \cite{LSV}. Section~\ref{Main results} contains a concise summary of the main results of the paper. In Section~\ref{The Levi-Civita phase function} we introduce a special phase function, the \emph{Levi-Civita phase function}, which will later act as the key ingredient of our geometric analysis, and analyse its properties in detail. A global invariant definition of the full symbol of the wave propagator is formulated in Section~\ref{The global invariant symbol of the propagator}, and an algorithm for the calculation of all its homogeneous components is provided. Some of the more technical material used in  Section~\ref{The global invariant symbol of the propagator} has been moved to a separate Appendix~\ref{The amplitude-to-symbol operator}.
In order to implement the algorithm presented in Section~\ref{The global invariant symbol of the propagator} one also needs to study invariant representations of the identity operator in the form of an oscillatory integral: this is the subject of Section~\ref{Invariant representation of the identity operator}. Section~\ref{The subprincipal symbol of the propagator} is devoted to a detailed study of the $g$-subprincipal symbol of the wave propagator, culminating with Theorem~\ref{theorem ODE subprincipal symbol} which gives an explicit formula for it. In Section~\ref{Small time expansion for the subprincipal symbol} we provide an explicit small time asymptotic expansion for the $g$-subprincipal symbol. This allows us to recover, as a by-product, the third Weyl coefficient, see Appendix~\ref{Weyl coefficients}. In Section~\ref{Explicit examples} we apply our construction to two explicit examples in 2D: the sphere and the hyperbolic plane. Finally, in Section~\ref{Circumventing topological obstructions: a geometric picture} we discuss in detail the issue of circumventing  obstructions due to caustics.

\section{Lagrangian manifolds and Hamiltonian flows}
\label{Lagrangian manifolds and Hamiltonian flows}

The theory of Fourier integral operators, beautifully set out in the seminal papers by H\"ormander and Duistermaat \cite{Ho71, DuHo}, proved to be an extremely powerful tool in the analysis of partial differential equations and gave rise to several flourishing lines of research still active nowadays. As it is unrealistic to give a concise account of such a vast field of mathematical analysis, we refer the interested reader to the aforementioned papers and to the monographs by Duistermaat \cite{Dui}, Tr\`eves \cite[Vol.~2]{Tr} and H\"ormander \cite[Vol.~4]{Hor} for a detailed exposition.

In this section we will briefly summarise the theory of global Fourier integral operators with complex-valued phase function as developed by Laptev, Safarov and Vassiliev \cite{LSV}, in a formulation adapted to the current paper. Here and further on we adopt the notation $T'M:=T^*M\setminus \{0\}$. 

We call \emph{Hamiltonian} any smooth function $h:T'M \to \mathbb{R}$ positively homogeneous in momentum of degree one, i.e.\ such that $h(x,\lambda \,\xi)=\lambda\, h(x,\xi)$ for every $\lambda>0$. For any such Hamiltonian, we denote by $(x^*(t;y,\eta), \xi^*(t;y,\eta))$ the Hamiltonian flow, namely the (global) solution of Hamilton's equations
\begin{equation}\label{Hamiltonian flow}
\begin{split}
\dot{x}^*(t;y,\eta)&=h_\xi(x^*(t;y,\eta),\xi^*(t;y,\eta)), \\ 
\dot{\xi}^*(t;y,\eta)&=-h_x(x^*(t;y,\eta),\xi^*(t;y,\eta))\,,
\end{split}
\end{equation}
with initial condition $(x^*(0;y,\eta), \xi^*(0;y,\eta))=(y,\eta)$. Observe that, as a consequence of \eqref{Hamiltonian flow}, $x^*$ and $\xi^*$ are positively homogeneous in momentum of degree zero and one respectively. Further on, whenever $x^*$ and $\xi^*$ come without argument, $(t;y,\eta)$ is to be understood. This will be done for the sake of readability when there is no risk of confusion.

The Hamiltonian flow, in turn, defines a Lagrangian submanifold $\Lambda_h$ of $T^*\mathbb{R}\times T'M\times T'M$ given by
\begin{equation}\label{Lambda h}
\Lambda_h:=\{   (t,-h(y,\eta)),(x^*(t;y,\eta),\xi^*(t;y,\eta)),(y,-\eta)\,\,|\,\, t\in \mathbb{R},\,(y,\eta)\in T'M \}.
\end{equation}
Indeed, a straightforward calculation shows that the canonical symplectic form $\omega$ on $T^*\mathbb{R}\times T'M\times T'M$ satisfies $\omega|_{\Lambda_h}=0$.

We call \emph{phase function} an infinitely smooth function 
\[
\varphi:\mathbb{R}\times M \times T'M \to \mathbb{C}
\]
which is non-degenerate, positively homogeneous in momentum of degree one and such that $\operatorname{Im}\varphi \ge 0$. We say that a phase function $\varphi$ \emph{locally parame\-terises} the submanifold $\Lambda_h$ if, in local coordinates $x$ and $y$ and in a neighbourhood of a given point of $\Lambda_h$, we have
\begin{equation*}
\Lambda_h=\{ (t,\varphi_t(t,x;y,\eta)),(x,\varphi_x(t,x;y,\eta)),(y,\varphi_y(t,x;y,\eta))\,\,|\,\,(t,x;y,\eta)\in \mathfrak{C}_\varphi \},
\end{equation*}
where $\mathfrak{C}_\varphi:=\{(t,x;y,\eta)\,\,|\,\,\varphi_\eta(t,x;y,\eta)=0 \}$. 

The above definitions allow us to say what it means for a distribution (in the sense of distribution theory, see \cite[Vol.~1]{Hor}) to be associated with $\Lambda_h$. A distribution $u$ is called a \emph{Lagrangian distribution of order m associated with $\Lambda_h$} if $u$ can be represented locally, modulo $C^\infty$, as the sum of oscillatory integrals of the form
\begin{equation*}
\mathcal{I}_\varphi(a)=\int \er^{\ir \,\varphi(t,x;y,\eta)}\, a(t,x;y,\eta)\,\dbar \eta
\end{equation*}
where $\varphi$ is a phase function locally parameterising $\Lambda_h$ and
$a\in S^m_{\mathrm{ph}}(\mathbb{R}\times M \times T'M)$ is a poly\-homogeneous function of order $m$.
Here and further on
\begin{equation}
\label{definition of dbar}
\dbar\eta=(2\pi)^{-d}\dr\eta.
\end{equation}
We recall that a poly\-homogeneous function of order $m$ is an infinitely smooth function
\[
a: \mathbb{R}\times M \times T'M \to \mathbb{C}
\]
admitting an asymptotic expansion in positively homogeneous components, i.e.
\begin{equation}
\label{asymp exp ampl}
a(t,x;y,\eta)\sim \sum_{k=0}^\infty a_{m-k}(t,x;y,\eta),
\end{equation}
where $a_{m-k}$ is positively homogeneous in $\eta$ of degree $m-k$. Here and in the following it is understood that whenever we write $S^m_{\mathrm{ph}}(E\times T'M)$ we mean polyhomogeneous functions of order $m$ on $T'M$ depending smoothly on the variables in $E$.

In the theory of Fourier integral operators the function $a$ is usually referred to as amplitude of the oscillatory integral. In the current paper we will call it \emph{amplitude} and denote it by a Roman letter, e.g.\ $a(t,x;y,\eta)$, when it depends on the variable $x\in M$, whereas we will call it \emph{symbol} and denote it by a fraktur letter, e.g.\ $\mathfrak{a}(t;y,\eta)$, when it is independent of the variable $x\in M$. In fact, as it will be explained in the following, one can always assume to be in the latter situation, modulo an infinitely smooth error in an appropriate sense.

It is a well known fact that with a real-valued phase function one can achieve the above mentioned parameterisation for a generic Lagrangian manifold only locally. Indeed, classical constructions involving global Fourier integral operators, see, for instance, \cite{Ho71}, \cite[Vol.~2]{Tr}, always resort to (the sum of) local oscillatory integrals. This is due to obstructions of topological nature represented on the one hand by the non-triviality of a certain cohomology class in $H^1(\Lambda_h,\mathbb{Z})$  \cite{Lees}, known as the \emph{Maslov class}, and on the other hand by the presence of caustics. In the case of a Lagrangian manifold generated by a homogeneous Hamiltonian flow the former obstruction is not present. The adoption of a complex-valued phase functions allows one to circumvent the latter and perform a construction which is inherently global.

To explain why this is the case, we first need to impose a restriction on the class of admissible phase functions. In particular, since our goal is to parameterise Lagrangian manifolds generated by a Hamiltonian, we need to impose compatibility conditions between our phase function and the Hamiltonian flow.

\begin{definition}[Phase function of class $\mathcal{L}_h$]\label{phase function of class Fh}
We say that a phase function $\varphi=\varphi(t,x;y,\eta)$ defined on $\mathbb{R}\times M\times T'M$ is \emph{of class} $\mathcal{L}_h$ if it satisfies the conditions
\begin{enumerate}[(i)]
\item $\left.\varphi\right|_{x=x^*}=0$,
\item $\left.\varphi_{x^\alpha}\right|_{x=x^*}=\xi_\alpha^*$,
\item $\left.\det\varphi_{x^\alpha\eta_\beta}\right|_{x=x^*}\ne0$,
\item $\operatorname{Im}\varphi\ge0$.
\end{enumerate}
\end{definition}
\noindent The space of phase functions of class $\mathcal{L}_h$ is non-empty and path-connected \cite[Lemma~1.7]{LSV}.  

We are now able to state the main result contained in \cite{LSV}.

\begin{theorem} \label{main theorem LSV}
The Lagrangian submanifold $\Lambda_h$ can be globally parameterised by a single phase function of class $\mathcal{L}_h$.
\end{theorem}

Theorem \ref{main theorem LSV} is crucial for the problem we want to study. In fact, take $h$ to be the principal symbol of the pseudodifferential operator $\sqrt{-\Delta}$, namely
\begin{equation}\label{hamiltonian laplacian}
h(x,\xi):=\left( g^{\alpha\beta}(x)\,\xi_\alpha\,\xi_\beta \right)^{1/2}.
\end{equation}
Then the flow \eqref{Hamiltonian flow} is (co)geodesic and the propagator for our hyperbolic PDE \eqref{main PDE} is a Fourier integral operator whose Schwartz kernel \eqref{propagator} is a Lagrangian distribution of order zero associated with the Lagrangian manifold $\Lambda_h$. As already noticed by Laptev, Safarov and Vassiliev in \cite{LSV}, being able to globally parameterise $\Lambda_h$ by a phase function of class $\mathcal{L}_h$ amounts to being able to write $u(t,x,y)$ as a single oscillatory integral, global both in space and in time.

This is not the only simplification brought about by this framework. Since the Maslov class of $\Lambda_h$ is trivial, and so is the reduced Maslov class, one can canonically identify sections of the Keller--Maslov bundle with smooth functions on $T'M$. In particular, the principal symbol of the Fourier integral operator defined by our Lagrangian distribution is simply the component of the highest degree of homogeneity $\mathfrak{a}_m$ in the asymptotic expansion of the symbol. We stress the fact that $\mathfrak{a}_m$ is a smooth scalar function on $T'M$ --- possibly depending on additional parameters --- which is independent of the choice of the phase function $\varphi$. Components of lower degree of homogeneity will generally depend on the choice of the phase function.

The crucial condition that allows us to pass through caustics is (iii) in Definition \ref{phase function of class Fh}. The degene\-racy of 
\begin{equation}\label{phixeta}
\left. \varphi_{x^\alpha\eta_\beta}\right|_{x=x^*}
\end{equation} for real-valued phase functions in the presence of conjugate points is what causes the analytic machinery to break down. The introduction of an imaginary part in $\varphi$ serves the purpose of ensuring that $\left.\det \varphi_{x^\alpha\eta_\beta}\right|_{x=x^*}\neq 0$ for all times. This is more than just a technical requirement, though; the object \eqref{phixeta} is actually capable of detecting information of topological nature about paths in $\Lambda_h$. This is reflected in the fact that, as it was firstly observed by Safarov and later formalised in \cite{LSV, SaVa}, \eqref{phixeta} is the core of a purely analytic definition of the Maslov index.

Consider the differential 1-form
\begin{equation}\label{Maslov form}
\vartheta_\varphi=- \dfrac{1}{2\pi} \dr \left[  \arg \left(\left. \det \varphi_{x^\alpha\eta_\beta}\right|_{x=x^*}  \right)^2\right]\,.
\end{equation}
Let $\gamma:=\{(x^*(t;y,\eta),\xi^*(t;y,\eta))\,\,|\,\, 0\le t \le T \}$ be a $T$-periodic Hamiltonian trajectory such that $x^*_\eta(T;y,\eta)=0$. Then the Maslov index of $\gamma$ is defined by
\begin{equation}\label{Maslov index}
\operatorname{ind} (\gamma):= \int_\gamma \, \vartheta_\varphi\,.
\end{equation}
It is easy to see that $\operatorname{ind} (\gamma)$ does not depend on the choice of the phase function $\varphi$. In fact, the index $\operatorname{ind}(\gamma)$ is determined by the de Rham cohomology class of $\vartheta_\varphi$ and \eqref{Maslov index} is the differential counterpart under the standard isomorphism between \v Cech and de Rham cohomologies of the approach in terms of cocycles adopted in \cite{Ho71}. See \cite[Section 1.5]{SaVa} for additional details.

\section{Main results}
\label{Main results}

We seek the Schwartz kernel \eqref{propagator}
of the propagator
\eqref{definition of propagator}
in the form
\begin{equation}\label{wave kernel = lagrangian + smoothing}
u(t,x,y)=\mathcal{I}_\varphi(\mathfrak{a})+\mathcal{K}(t,x,y),
\end{equation}
where $\mathcal{K}$ is an infinitely smooth kernel and
\begin{equation}
\label{main oscillatory integral}
\mathcal{I}_\varphi(\mathfrak{a})=
\int_{T^*_yM}\er^{\ir \varphi(t,x;y,\eta;\epsilon)}
\,\mathfrak{a}(t;y,\eta;\epsilon)\,\chi(t,x;y,\eta)
\,w(t,x;y,\eta;\epsilon)\,\dbar\eta
\end{equation}
is a global oscillatory integral.
Here $\varphi$ is a particular phase function of class $\mathcal{L}_h$,
with $h$ given by \eqref{hamiltonian laplacian},
which will be introduced in Section~\ref{The Levi-Civita phase function}.
This phase function is completely determined by the metric
and a positive parameter~$\epsilon$ and will be called the
{Levi-Civita phase function}.
Rigorous definitions of the symbol
$\mathfrak{a}$, cut-off $\chi$ and weight~$w$ appearing
on the RHS of \eqref{main oscillatory integral}
will be provided in Section~\ref{The global invariant symbol of the propagator}.
Let us emphasise that the representation \eqref{main oscillatory integral}
will be global in time $t\in\mathbb{R}$ and in space $x,y\in M$.

\

Our main results are as follows.

\begin{enumerate}[1.]
\item
We provide an invariant definition of the full symbol of the wave propagator
as a scalar function $\mathfrak{a}(t;y,\eta;\epsilon)$,
\[
\mathfrak{a}:
\mathbb{R}\times T'M\times \mathbb{R}_+ \to \mathbb{C},
\]
along with an explicit algorithm for the calculation
of all its homogeneous components,
see Section~\ref{The global invariant symbol of the propagator}.
Throughout the paper we use the notation $\mathbb{R}_+:=(0,+\infty)$.

\item
We determine the symbol of the identity operator
written as an invariant oscillatory integral,
see Section~\ref{Invariant representation of the identity operator}.

\item
We perform a detailed study of the
$g$-subprincipal symbol of the propagator
and provide a simplified algorithm for its calculation,
see Section~\ref{The subprincipal symbol of the propagator}.

\item
We write down a
small time asymptotic formula for the $g$-subprincipal symbol of the propagator,
see Theorem~\ref{theorem small time}.
\item
We apply our construction to
maximally symmetric
spaces of constant curvature in 2D,
the standard 2-sphere and the hyperbolic plane, see Section~\ref{Explicit examples}.
\item
Using our complex-valued phase function,
we provide a geometric construction which allows us
to visualise the analytical circumvention of
obstructions due to caustics, see Theorem~\ref{theorem geometric characterisation phixeta}.
\end{enumerate}

\section{The Levi-Civita phase function}
\label{The Levi-Civita phase function}

In this section we will introduce a distinguished phase function, the Levi-Civita phase function, providing motivation and basic properties.

\begin{definition}[Levi-Civita phase function]
We call the \emph{Levi-Civita phase function} the infinitely smooth function
\[
\varphi: \mathbb{R}\times M \times T'M \times \mathbb{R}_+ \to \mathbb{C}
\]
defined by 
\begin{equation}\label{phase function with parallel transport}
\varphi(t,x;y,\eta;\epsilon):= \int_\gamma \zeta \, \dr z + \dfrac{\ir\,\epsilon}{2}\, h(y,\eta) \, \operatorname{dist}^2(x,x^*(t;y,\eta))
\end{equation}
when $x$ lies in a geodesic neighbourhood\footnote{
Here and further on by `geodesic neighbourhood of $z$' we mean the image under the exponential map $\exp_{z}:T_zM \to M$ of a star-shaped neighbourhood $\mathcal{V}$ of $0\in T_{z}M$ such that $\exp_z|_{\mathcal{V}}$ is a diffeomorphism.}
of $x^*(t;y,\eta)$
and continued smoothly elsewhere in such a way that
$\operatorname{Im}\varphi\ge0$. 
The function $\operatorname{dist}$ is the Riemannian geodesic distance, the path of integration $\gamma$ is the (unique) shortest geodesic connecting $x^*(t;y,\eta)$ to $x$, and $\zeta$ is the result of the parallel transport of $\xi^*(t;y,\eta)$ along~$\gamma$.
\end{definition}

\begin{figure}[!h]
\begin{center}
\begin{tikzpicture}[scale=0.6]

\path[name path=border1] (0,0) to[out=-10,in=150] (6,-2);
\path[name path=border2] (12,1) to[out=150,in=-10] (5.5,3.2);
\path[name path=redline] (0,-0.4) -- (12,1.5);
\path[name intersections={of=border1 and redline,by={a}}];
\path[name intersections={of=border2 and redline,by={b}}];

\draw[black] 
  (0,0) to[out=-10,in=150] (6,-2) -- (12,1) to[out=150,in=-10] (5.5,3.7) -- cycle;
  
\draw[black]
(2.5,0.4) to[out=-1] (8,1);

\draw[black,dashed] (8,1) circle (1.3);

\path (2.5,0.4) node[label= left :\(y\)] (q1) {$\cdot$};
\path (8,1) node[label= right :$x^\ast$] (q1) {$\cdot$};

\node[rotate=0] at (6.15,-1.35) {$M$};

\end{tikzpicture}
\end{center}
\end{figure}
The imaginary part of $\varphi$ is pre-multiplied by a positive parameter $\epsilon$ in order to keep track of the effects of making $\varphi$ complex-valued. The real-valued case can be recovered by setting $\epsilon=0$.

It is straightforward to check that the Levi-Civita phase function $\varphi$ is of class $\mathcal{L}_h$. Note that in geodesic normal coordinates $x$ centred at
$x^*(t;y,\eta)$ the function $\varphi$ reads locally
\begin{equation}\label{phase function in normal coordinates}
\begin{split}
\varphi(t,x;y,\eta;\epsilon)&=(x-x^*(t;y,\eta))^\alpha\,\xi^*_\alpha(t;y,\eta) \\
&+ \dfrac{\ir\,\epsilon}{2}\, h(y,\eta) \, 
\delta_{\mu\nu}\,(x-x^*(t;y,\eta))^\mu(x-x^*(t;y,\eta))^\nu.
\end{split}
\end{equation}
Our phase function is invariantly defined
and naturally dictated by the geometry of $(M,g)$.
Its construction relies on the use of the Levi-Civita connection associated with the Riemannian metric $g$, which justifies its name.
From the analytic point of view, the adoption of the Levi-Civita phase function is particularly convenient in that it turns the Laplace-Beltrami operator into a partial differential operator with almost constant coefficients, up to curvature terms. In a sense, $\varphi$ `straightens out' the geometry of $(M,g)$, thus bringing about considerable simplifications in the analysis.
More precisely, the Levi-Civita phase function with $\epsilon=0$
has the following
properties which a general phase function associated with
the geodesic flow
does not possess:
\begin{enumerate}[(i)]
\item
$\left.(\Delta\varphi)\right|_{x=x^*}=0$;
\item
$\left.(\varphi_{tt})\right|_{x=x^*}=0$;
\item
the full symbol of the identity operator is $1$,
see Lemma~\ref{symbol identity Levi-Civita epsilon=0}.
\end{enumerate}

\begin{remark}
The real-valued Levi--Civita phase function appears, in various forms, in \cite{LSV}, \cite{SaVa} and \cite{McSa}. Note, however, that the geometric phase function used in the parametrix construction in \cite{Zel4} and \cite{canzani} is not the same as \eqref{phase function with parallel transport} for $\epsilon=0$: the phase function appearing in \cite{Zel4} and \cite{canzani} is \emph{linear} in $t$, whereas ours is not. This is essentially due to the fact that the Levi-Civita  phase function is constructed out of the cogeodesic flow.
\end{remark}

\begin{lemma}\label{lemma about versions of Re PF}
We have
\begin{equation*}\label{real part of PF with velocity}
\int_\gamma\zeta\,\dr z=\langle\xi^*(t;y,\eta),\exp^{-1}_{x^*}(x)\rangle,
\end{equation*}
where $\exp$ denotes the exponential map and
$\langle\,\cdot\,,\,\cdot\,\rangle$
is the (pointwise) canonical pairing between cotangent and tangent bundles.
\end{lemma}

\begin{proof}
Denoting by $P_{\gamma(s)}:T^*_{x^*(t;y,\eta)}M\to T^*_{\gamma(s)}M$ the one-parameter family of operators realising the parallel transport of covectors from $x^*(t;y,\eta)$ to $\gamma(s)$ along $\gamma:[0,1]\to M$, we have
\begin{equation*}
\begin{split}
\int_\gamma \zeta\,\dr z &= \int_0^1 \langle P_{\gamma(s)}(\xi^*(t;y,\eta)),\dot\gamma(s) \rangle\, \dr s\\ 
&=\int_0^1 \langle \xi^*(t;y,\eta),\dot\gamma(0) \rangle\, \dr s=\langle\xi^*(t;y,\eta),\exp_{x^*}^{-1}(x)\rangle,
\end{split}
\end{equation*}
where the dot stands for the derivative with respect to the parameter $s$.
At the second step we used the fact that
\begin{equation*}
\begin{split}
&\dfrac{d}{ds}\langle P_{\gamma(s)}(\xi^*(t;y,\eta)),\dot\gamma(s) \rangle \\&\quad = \langle \nabla_{\dot \gamma(s)}\,P_{\gamma(s)}(\xi^*(t;y,\eta)),\dot\gamma(s) \rangle +\langle P_{\gamma(s)}(\xi^*(t;y,\eta)),\nabla_{\dot \gamma(s)}\dot\gamma(s) \rangle
=0.
\end{split}
\end{equation*}
\end{proof}
In view of Lemma \ref{lemma about versions of Re PF}, we can recast the Levi-Civita phase function \eqref{phase function with parallel transport} in the more explicit form
\begin{equation}\label{phase function with distance}
\varphi(t,x;y,\eta;\epsilon):= -\frac{1}{2} \langle{\xi^*},{\left. \operatorname{grad}_z [\operatorname{dist}^2(x,z)]\right|_{z=x^*}}\rangle
+\dfrac{\ir \, \epsilon }{2} \,h(y,\eta) \,
\operatorname{dist}^2(x^\ast,x)\,,
\end{equation}
where the initial velocity $\exp_{x^*}^{-1}(x)$ is expressed in terms of the geodesic distance squared.

As briefly discussed in Section~\ref{Lagrangian manifolds and Hamiltonian flows}, the phase function is capable of detecting information of topological nature. In particular, a crucial role is played by the two-point tensor
$\varphi_{x^\alpha \eta_\beta}$ and its determinant.

\begin{theorem}\label{theorem phixeta at x=xstar}
In any coordinate systems $x$ and $y$,
$\,\varphi_{x^\alpha\eta_\beta}$ along the flow is given by
\begin{equation}\label{phixeta at x=xstar}
\left.\varphi_{x^\alpha\eta_\beta}\right|_{x=x^*}= 
\dfrac{\partial \xi^\ast_\alpha}{\partial \eta_\beta} 
- \Gamma^\mu{}_{\alpha\nu}(x^*)\,\xi^*_\mu\,\dfrac{\partial x^*{}^\nu}{\partial \eta_\beta} 
- \ir \,\epsilon\, h(y,\eta) \,g_{\alpha\nu}(x^*)\,\dfrac{\partial x^*{}^\nu}{\partial \eta_\beta}\,,
\end{equation}
where $\Gamma^\mu{}_{\alpha\nu}$ are the Christoffel symbols.
\end{theorem}

\begin{proof}
Let us seek an expansion for the phase function $\varphi$ in powers of $(x-x^*)$ up to second order. To this end, we need to obtain an analogous expansion for $\dot\gamma(0)$ first. Recall that $\gamma:[0,1]\to M$ is the shortest geodesic connecting $x^*$ to $x$, hence satisfying 
\[
\gamma(0)=x^*, \qquad \gamma(1)=x.
\]
Put
\[
\gamma(s)=x^*+(x-x^*)\,s+ z(s;x,x^*),
\]
where $z$ is a correction of order $O(\|x-x^*\|^2)$ such that $z(0)=0$ and $z(1)=0$. By requiring $\gamma$ to satisfy the geodesic equation, we obtain
\[
\ddot z(s)+ \Gamma^{\alpha}{}_{\mu\nu}(\gamma(s))\,(x-x^*)^\mu\,(x-x^*)^\nu=0+
O(\|x-x^*\|^3),
\]
from which we get
\[
z(s)=\frac{s(1-s)}{2}\,\Gamma^\alpha{}_{\mu\nu}(x^*)\,(x-x^*)^\mu\,(x-x^*)^\nu+
O(\|x-x^*\|^3)
\]
and, in turn,
\[
\dot \gamma^\alpha(0)=(x-x^*)^\alpha +\frac12\,\Gamma^\alpha{}_{\mu\nu}(x^*)\,(x-x^*)^\mu\,(x-x^*)^\nu+
O(\|x-x^*\|^3).
\]
It ensues that the Levi-Civita phase function admits the expansion
\begin{equation*}\label{expansion PF second order in x}
\begin{split}
\varphi(t,x;y,\eta;\epsilon)=& (x-x^*)^\alpha\,\xi^*_\alpha +\frac12 \,\Gamma^\alpha{}_{\mu\nu}(x^*)\,\xi^*_\alpha\,(x-x^*)^\mu\,(x-x^*)^\nu \\
&+ \dfrac{\ir\,\epsilon\,h(y,\eta)}{2} g_{\mu\nu}(x^*)\,(x-x^*)^\mu\,(x-x^*)^\nu +
O(\|x-x^*\|^3).
\end{split}
\end{equation*}
Formula \eqref{phixeta at x=xstar} now follows by direct differentiation.
\end{proof}

The explicit formula established in Theorem~\ref{theorem phixeta at x=xstar} is quite useful. In fact, it offers a direct way of investigating the properties of $\Lambda_h$ and computing the Maslov index. We will come back to this later on.

\section{The global invariant symbol of the propagator}
\label{The global invariant symbol of the propagator}

In this section we will present an algorithm for the construction of a
global invariant full symbol $\mathfrak{a}$ for the wave propagator. 

In view of formulae
\eqref{wave kernel = lagrangian + smoothing}
and
\eqref{main oscillatory integral},
let us consider the Lagrangian distribution
\begin{equation}\label{Lagrangian distribution for wave}
\mathcal{I}_\varphi(\mathfrak{a})=\int_{T^*_yM}\er^{\ir \varphi(t,x;y,\eta;\epsilon)}
\,\mathfrak{a}(t;y,\eta;\epsilon)\,\chi(t,x;y,\eta)
\,w(t,x;y,\eta;\epsilon)\,\dbar\eta\,,
\end{equation}
where the quantities on the RHS are defined as follows.
\begin{itemize}
\item $\varphi$ is the Levi-Civita phase function \eqref{phase function with distance}.
\item
$\mathfrak{a}\in S^0_{\mathrm{ph}}(\mathbb{R}\times T'M\times\mathbb{R}_+)$
is a polyhomogeneous symbol with asymptotic expansion
\begin{equation}\label{homogeneous asymptotic expansion of symbol}
\mathfrak{a}(t;y,\eta;\epsilon)\sim \sum_{k=0}^\infty \mathfrak{a}_{-k}(t;y,\eta;\epsilon),
\end{equation}
where the $\mathfrak{a}_{-k}\in S^{-k}(\mathbb{R}\times T'M\times\mathbb{R}_+)$
are positively homogeneous in momentum of degree $-k$. They represent the unknowns of our construction.
\item $\chi\in C^\infty(\mathbb{R}\times M\times T'M)$ is a cut-off satisfying the requirements
\begin{enumerate}[(i)]
\item $\chi(t,x;y,\eta)=0$ on $\{(t,x;y,\eta) \,|\, |h(y,\eta)|\leq 1/2\}$;
\item $\chi(t,x;y,\eta)=1$ on the intersection of $\{(t,x;y,\eta) \,|\, |h(y,\eta)| \geq 1\}$ with some conical neighbourhood of $\{(t,x^\ast(t;y,\eta);y,\eta) \}$;
\item $\chi(t,x;y,\alpha\, \eta)=\chi(t,x;y,\eta)$ for $\alpha\geq 1$ on $\{ (t,x;y,\eta) \, | \, |h(y,\eta)|\geq 1   \}$.
\end{enumerate}
The function $\chi$ serves the purpose of localising the domain of integration to a neighbourhood of orbits with initial conditions $(y,\eta)\in T'M$ and away from the zero section. Recall that the Hamiltonian $h$ is positively homogeneous in $\eta$ of degree $1$. Further on, we will set $\chi\equiv 1$ while carrying out calculations. This will not affect the final result, as stationary phase arguments show that contributions to the oscillatory integral \eqref{Lagrangian distribution for wave} only come from
a neighbourhood of the set
\[
\{(t,x;y,\eta)\,\,|\,\, x=x^*(t;y,\eta) \}
\]
on which $\varphi_\eta=0$.
Different choices of $\chi$ result in oscillatory integrals differing by infinitely smooth contributions.
\item $w(t,x;y,\eta;\epsilon)$ is defined by
\begin{equation}\label{weight w}
w(t,x;y,\eta;\epsilon):=
[\rho(x)]^{-1/2}\,[\rho(y)]^{-1/2}
\left[
{\det}^2\!
\left(
\varphi_{x^\alpha \eta_\beta}(t,x;y,\eta;\epsilon)
\right)
\right]^{1/4}
\end{equation}
with $\rho=\sqrt{\det g_{\mu\nu}}$  as in \eqref{riemannian density}.
The branch of the complex root is chosen in such a way that 
 \[
 \left.
\arg \left[
{\det}^2\!
\left(
\varphi_{x^\alpha \eta_\beta}(t,x;y,\eta;\epsilon)
\right)
\right]^{1/4} \right|_{t=0} =0.
\]
The existence of a smooth
global branch whose argument turns to zero at $t=0$ was established by \cite[Lemma 3.2]{LSV}. 
The weight $w$ is a $(-1)$-density in $y$ and a scalar function in all other arguments. It ensures that the oscillatory integral \eqref{Lagrangian distribution for wave} is a scalar and that the principal symbol $\mathfrak{a}_0$ of the wave propagator does not depend on the choice of the phase function \cite[Theorem 2.7.11]{SaVa}.
Thanks to condition (iii) in Definition \ref{phase function of class Fh} we can assume, without loss of generality, that $w$ is non-zero whenever $\chi$ is non-zero.
\end{itemize}

\begin{remark}
\label{remark on fourth root}
The reason we write
$\left[
{\det}^2\!
\left(
\varphi_{x^\alpha \eta_\beta}
\right)
\right]^{1/4}$
in formula \eqref{weight w}
rather than
$\sqrt{\det\varphi_{x^\alpha \eta_\beta}}\,$
is that the coordinate systems $x$ and $y$ may be different: inversion of a single coordinate $x^\alpha$
changes the sign of $\,\det\varphi_{x^\alpha \eta_\beta}\,$ and so does inversion of a single coordinate $y^\beta$.
\end{remark}

The general idea is to choose the phase function to be the Levi-Civita phase function, fixing it once and for all, and to seek a formula for the corresponding scalar symbol
$\mathfrak{a}$. This is achieved by means of the following algorithm, which reduces the problem of solving partial differential equations to the much simpler problem of solving ordinary differential equations.

\

{\bf Step one}.
Set $\chi(t,x;y,\eta;\epsilon)=1$ and apply the wave operator
\begin{equation}
\label{definition of the wave operator}
\mathcal{P}:=\partial_t^2-\Delta_x
\end{equation}
to \eqref{Lagrangian distribution for wave}. 
The result is an oscillatory integral 
\begin{equation}\label{unreduced oscillatory integral with a}
\mathcal{I}_\varphi(a)=\mathcal{P}\,\mathcal{I}_\varphi(\mathfrak{a})
\end{equation}
of the same form but with a different amplitude
\[
\begin{split}
&a(t,x;y,\eta;\epsilon)\\
&\quad=\er^{-\ir \varphi(t,x;y,\eta;\epsilon)}\,
[w(t,x;y,\eta;\epsilon)]^{-1}\,
\mathcal{P} 
\left(  \er^{\ir \varphi(t,x;y,\eta;\epsilon)} \,\mathfrak{a}(t;y,\eta;\epsilon)\,w(t,x;y,\eta;\epsilon)
 \right).
\end{split}
\]
Observe that $a\in S^2_{\mathrm{ph}}(\mathbb{R}\times M\times T'M\times \mathbb{R}_+)$. The use of the full wave operator $\mathcal{P}$ as opposed to the half-wave operator $(-\ir\,\partial_t+\sqrt{-\Delta})$ is justified by \cite[Theorem 3.2.1]{SaVa}.

\

{\bf Step two}. Construct a new oscillatory integral with $x$-independent amplitude $\mathfrak{b}=\mathfrak{b}(t;y,\eta;\epsilon)$, coinciding with \eqref{unreduced oscillatory integral with a} up to an infinitely smooth term:
\begin{equation}\label{identity of amplitude reduction}
\mathcal{I}_\varphi(\mathfrak{b}) \overset{\mod C^\infty}{=} \mathcal{I}_\varphi(a).
\end{equation}
Such a procedure is called \emph{reduction of the amplitude}.
This can be done by means of special operators, as described below.

Put
\begin{equation}\label{operator L}
L_\alpha:=\left[(\varphi_{x\eta})^{-1}\right]_\alpha{}^\beta\,\dfrac{\partial}{\partial x^\beta}
\end{equation}
and define
\begin{subequations}\label{operators mathfrak S}
\begin{gather}
\mathfrak{S}_0:=\left.\left( \,\cdot\, \right)\right|_{x=x^*}\,, \label{mathfrak S0}\\
\mathfrak{S}_{-k}:=\mathfrak{S}_0 \left[ \ir \, w^{-1} \frac{\partial}{\partial \eta_\beta}\, w \left( 1+ \sum_{1\leq |\boldsymbol{\alpha}|\leq 2k-1} \dfrac{(-\varphi_\eta)^{\boldsymbol{\alpha}}}{\boldsymbol{\alpha}!\,(|\boldsymbol{\alpha}|+1)}\,L_{\boldsymbol{\alpha}} \right) L_\beta  \right]^k\,.\label{mathfrak Sk}
\end{gather}
\end{subequations}
Bold Greek letters in \eqref{mathfrak Sk} denote multi-indices in $\mathbb{N}^d_0$, $\boldsymbol{\alpha}=(\alpha_1, \ldots, \alpha_d)$, $|\boldsymbol{\alpha}|=\sum_{j=1}^d \alpha_j$ and $(-\varphi_\eta)^{\boldsymbol{\alpha}}:=(-1)^{|\boldsymbol{\alpha}|}\, (\varphi_{\eta_1})^{\alpha_1}\dots ( \varphi_{\eta_d})^{\alpha_d}$. All differentiations are applied to the whole expression to the right of them. The operator
\eqref{mathfrak Sk} is well defined
because the differential operators $L_\alpha$ commute, see Lemma~\ref{lem:commutativity_Ls} in Appendix~\ref{The amplitude-to-symbol operator}.

When applied to a homogeneous function, the operator $\mathfrak{S}_{-k}$ decreases the degree of homogeneity in $\eta$ by $k$. Hence, denoting by
$
a \sim \sum_{j=0}^\infty a_{2-j}
$
the asymptotic polyhomogeneous expansion of $a$, the homogeneous components of the symbol $\mathfrak{b}$ are
\begin{equation}\label{construction homogeneous components of the reduced amplitude}
\mathfrak{b}_l:=\sum_{2-j-k=l} \mathfrak{S}_{-k}\,a_{2-j},\qquad
l=2,1,0,-1,\ldots.
\end{equation}
We call the operator $\mathfrak{S} \sim \sum_{k=0}^{\infty} \mathfrak{S}_{-k}\,$ the \emph{amplitude-to-symbol operator}. It maps the $x$-dependent amplitude $a$ to the $x$-independent symbol $\mathfrak{b}$. The construction of $\mathfrak{S}$ and the proof of the equality \eqref{identity of amplitude reduction} are presented in Appendix \ref{The amplitude-to-symbol operator}.\\

{\bf Step three}. Impose the condition that our oscillatory integral
\eqref{Lagrangian distribution for wave}
satisfies the wave equation, namely
\[
\mathcal{P}\mathcal{I}_\varphi(\mathfrak{a})\overset{\mod C^\infty}{=} \mathcal{I}_\varphi(\mathfrak{b}) = 0.
\]
This is achieved by solving \emph{transport equations} obtained by equating to zero the homogeneous components of the reduced amplitude $\mathfrak{b}$:
\begin{equation}\label{transport equations}
\mathfrak{b}_l=0,\qquad
l=2,1,0,-1,\ldots.
\end{equation}
Note that the equation $\mathfrak{b}_{2}=0$ may be referred to as \emph{eikonal equation}, see \cite[subsection~2.4.2]{SaVa} for details.

Formula \eqref{transport equations} describes a hierarchy of ordinary differential equations in the variable $t$ whose unknowns are the homogeneous components of the original amplitude $\mathfrak{a}$. Solving such equations iteratively produces an explicit formula for the symbol of the wave kernel. Initial conditions
$\mathfrak{a}_{-k}(0;y,\eta;\epsilon)$ are established in such a way that
at $t=0$ our oscillatory integral
\eqref{Lagrangian distribution for wave}
is, modulo $C^\infty$, the integral kernel of the identity operator ---
see Section~\ref{Invariant representation of the identity operator}
for details. 

\begin{remark}
One knows \emph{a priori} that the leading homogeneous term in the expansion \eqref{homogeneous asymptotic expansion of symbol} is
\begin{equation}\label{principal symbol is 1}
\mathfrak{a}_0(t;y,\eta;\epsilon)=1.
\end{equation}
This is a consequence of the fact that the subprincipal
symbol of the Laplace--Beltrami operator is zero,
see \cite[Theorem~4.1]{LSV} or \cite[Theorem~3.3.2]{SaVa}.
Formula \eqref{principal symbol is 1}
holds for any choice of phase function
due to the way \eqref{Lagrangian distribution for wave} is designed.
\end{remark}

Let us explain more precisely what we mean by saying that our construction is global in time. The issue with the standard construction is that, in the presence of caustics, one cannot parameterise globally the Lagrangian manifold generated by the Hamiltonian flow of the principal symbol by means of a single real-valued phase function. In our analytic framework, this means that the phase function may become degenerate when $x=x^*(t;y,\eta)$ and $x^*(t;y,\eta)$ is in the cut locus or conjugate locus of $y$. In turn, the weight $w$ vanishes and the Fourier integral operator with integral kernel \eqref{Lagrangian distribution for wave} ceases to be well-defined. The adoption of a complex-valued phase function allows us to circumvent these problems and construct a Fourier integral operator which is always well defined.

Note that the issue of `local vs global' is \emph{not} related to the use of the geodesic distance in the definition of our phase function. In fact, what appears in our construction is the geodesic distance between $x$ and $x^*$. Now,  non-smoothing contributions come from points $x$ close to $x^*$, as these are the only stationary points for the phase in the support of the amplitude. As the injectivity radius is strictly positive, one can always choose a cut-off $\chi$ in such a way that the above points $x$ are not in the cut locus or conjugate locus of $x^*$. What happens outside a small open neighbourhood of the geodesic flow gives an infinitely smoothing contribution.

We are now in a position to give the following definition.

\begin{definition}
\label{invariant definition of symbol}
We define the \emph{symbol of the wave propagator} as the scalar function
\begin{gather*}
\mathfrak{a}: \mathbb{R} \times T'M \times \mathbb{R}_+ \to \mathbb{C}\,\nonumber \\
\mathfrak{a}(t;y,\eta;\epsilon) = 1 +\mathfrak{a}_{-1}(t;y,\eta;\epsilon)+\mathfrak{a}_{-2}(t;y,\eta;\epsilon)+\dots
\end{gather*}
obtained through the above algorithm
with the choice of the Levi-Civita phase function.
\end{definition}

The above definition is invariant: $\mathfrak{a}$ depends only on $\varphi$ which, in turn, arises from the geometry of $(M,g)$ in a coordinate-free, covariant manner.

The algorithm provided in this section allows us to construct the wave propagator as a Fourier integral operator whose Schwartz kernel is a global Lagrangian distribution, namely, a single oscillatory integral global in space and in time, with invariantly defined symbol. In particular, it allows one to circumvent at an analytic level obstructions arising from caustics.

In Section \ref{The subprincipal symbol of the propagator} we will see the algorithm in action and perform a detailed analysis of the $g$-subprincipal symbol. In Section \ref{Explicit examples} we will apply our algorithm to two explicit examples.

\begin{remark}
The remainder terms in the asymptotic formulae provided in this paper are not uniform in time: they are only uniform over finite time intervals. This is to be expected when working with Fourier integral operators.
\end{remark}

\begin{remark}[Scalar functions vs half-densities]
In microlocal analysis and spectral theory it is often convenient to work with operators acting on half-densities, as opposed to scalar functions. Our construction is easily adaptable to half-densities as follows.

\begin{itemize}
\item 
Replace the Laplacian on functions $\Delta$ with the corresponding operator on half-densities
\begin{equation*}
\label{laplacian on half-densities}
\widetilde{\Delta}=\rho(x)^{1/2}\,\Delta\,\rho(x)^{-1/2}.
\end{equation*}
\item
Replace the weight $w$ with
\begin{equation*}\label{tilde w}
\widetilde{w}=\left[
{\det}^2 \!
\left(
\varphi_{x^\alpha \eta_\beta}
\right)
\right]^{1/4}\,.
\end{equation*}
Note that $\widetilde{w}$ is now a $\frac12$-density in $x$ and a $\,-\frac12$-density in $y$.
\item Seek the integral kernel of the propagator as an oscillatory integral of the form
\begin{equation*}
\widetilde{\mathcal{I}}_\varphi(\mathfrak{a})=\int \er^{\ir \,\varphi}\,\mathfrak{a}\,\widetilde{w}\,\dbar \eta\,.
\end{equation*}
Note that $\widetilde{\mathcal{I}}_\varphi(\mathfrak{a})$ is a half-density both in $x$ and in $y$.
\item Carry out the above algorithm.
\end{itemize}
It can be shown that we end up with the same full symbol
of the wave propagator as when working with scalar functions.
\end{remark}

\begin{remark}
By carrying out the integration in $\eta$ in \eqref{Lagrangian distribution for wave} for $x$ sufficiently close to $y$ one obtains the well-known Hadamard expansion, see, e.g., \cite{berard} and \cite[Remark~2.5.5]{bar}. Our construction provides an explicit global version of the known local expansion.
\end{remark}

\section{Invariant representation of the identity operator}
\label{Invariant representation of the identity operator}

Step three of our algorithm described in
Section~\ref{The global invariant symbol of the propagator}
involves initial conditions determined by the symbol of the identity operator,
which appears in our construction
as a pseudo\-differential operator
written in the form
\begin{equation*}
\label{Lagrangian distribution for wave will disappear}
\int_{T'M}\er^{\ir \varphi(0,x;y,\eta;\epsilon)}
\,\mathfrak{s}(y,\eta;\epsilon)\,\chi(0,x;y,\eta)
\,w(0,x;y,\eta;\epsilon)\,
\,(\,\cdot\,)\,\rho(y)\,\dr y\,\dbar\eta
\end{equation*}
with the Levi-Civita phase function and some symbol $\mathfrak{s}$,
cf.~\eqref{Lagrangian distribution for wave}.
Recall that $\chi$ is a cut-off and $w$ is defined by formula \eqref{weight w}.
Note also that coordinate systems $x$ and $y$ may be different.

Invariant representation 
of pseudo\-differential operators on manifolds
is not a well studied subject.
Existing literature comprises
\cite{McSa}
and
\cite{DLS},
though invariant representations come there in slightly
different forms.
The aim of this section is to establish a few results in this direction
for the identity operator.

Clearly, the principal symbol of the identity operator is
\begin{equation}
\label{ principal symbol of the identity operator is 1}
\mathfrak{s}_0(y,\eta)=1,
\end{equation}
irrespective of the choice of the phase function. In general, one would expect subleading homogeneous components of the symbol to depend on the phase function. This turns out not to be the case for $\mathfrak{s}_{-1}$, which is zero for any choice of phase function.

\begin{theorem}\label{theorem subprincipal identity FIO}
Let $\phi\in C^{\infty}(M\times T'M;\mathbb{C})$ be a positively homogeneous function (in momentum) of degree 1 satisfying the conditions
\begin{enumerate}
\item[(a)] $\phi(x;y,\eta)=(x-y)^\alpha\,\eta_\alpha + O(\|x-y\|^2)$\,,
\item[(b)] $\operatorname{Im} \phi\geq 0$.
\end{enumerate}
In stating condition (a) we use the same local coordinates for $x$ and $y$.

Consider a pseudodifferential operator
\begin{equation}\label{identity as FIO}
(\mathfrak{I}_{\phi,s}\, f)(x) = \int_{T'M}
\er^{\ir\, \phi(x;y,\eta)}
\,\mathfrak{s}(y,\eta)
\,\chi(x;y,\eta)
\,v(x;y,\eta)
\,f(y)\,\dr y\,\dbar \eta\,,
\end{equation}
where $\mathfrak{s}\sim \sum_{k\in \mathbb{N}_0} \mathfrak{s}_{-k}\in S^0(T'M)$,
$\chi$ is a cut-off
and
\begin{equation}
\label{definition of v}
v(x;y,\eta)=\rho(x)^{-1/2}\,\rho(y)^{1/2} \,[{\det}^2 \phi_{x\eta}]^{1/4}\,.
\end{equation}
If $\,\mathfrak{I}_{\phi,s}-\operatorname{Id}$
is an infinitely smoothing operator,
then
\begin{equation*}
\label{subleading term for identity operator}
\mathfrak{s}_{-1}(y,\eta)=0.
\end{equation*}
\end{theorem}

\begin{remark}
It is easy to see that
the quantity defined by formula \eqref{definition of v} is a scalar function
$v:M\times T'M\to\mathbb{C}$. The branch of the complex root is chosen so
that $v=1$ on the diagonal $x=y$.
\end{remark}

\begin{proof}[Proof of Theorem~\ref{theorem subprincipal identity FIO}]
Let us define the dual pseudodifferential operator
$\mathfrak{I}_{\phi,s}'$
via the identity
\[
\int_M
\bigl[
 k(x) 
\bigr]
\bigl[
(\mathfrak{I}_{\phi,s}\, f)(x)
\bigr]
\,
\rho(x)
\,
\dr x
=
\int_M
\bigl[
(\mathfrak{I}_{\phi,s}'\, k)(y)
\bigr]
\bigl[
 f(y) 
\bigr]
\,
\rho(y)
\,
\dr y\,,
\]
where $f,k:M\to \mathbb{C}$ are smooth functions.
The explicit formula for the  pseudodifferential operator
$\mathfrak{I}_{\phi,s}'$ reads
\begin{equation*}\label{identity as FIO dual}
(\mathfrak{I}_{\phi,s}'\, k)(y) = \int_{M\times T_y'M}
\er^{\ir\, \phi(x;y,\eta)}
\,\mathfrak{s}(y,\eta)
\,\chi(x;y,\eta)
\,u(x;y,\eta)
\,k(x)\,\dr x\,\dbar \eta\,,
\end{equation*}
where
\begin{equation}
\label{definition of u}
u(x;y,\eta)
=
\rho(x)\,\rho(y)^{-1}\,v(x;y,\eta)
=
\rho(x)^{1/2}\,\rho(y)^{-1/2} \,[{\det}^2 \phi_{x\eta}]^{1/4}\,.
\end{equation}
Of course, the condition that
$\,\mathfrak{I}_{\phi,s}-\operatorname{Id}$
is an infinitely smoothing operator
is equivalent to the condition that
$\,\mathfrak{I}_{\phi,s}'-\operatorname{Id}$
is an infinitely smoothing operator.

Let us now fix an arbitrary point $P\in M$ and work in local coordinates $y$ such that $y=0$ at $P$.
Furthermore, let us use the same local coordinates for $x$ and for $y$.
Consider the map
\begin{equation}
\label{distribution 1}
k\mapsto(\mathfrak{I}_{\phi,s}'\, k)(0).
\end{equation}
The map \eqref{distribution 1} is a distribution,
a continuous linear functional.
We want the distribution \eqref{distribution 1} to approximate, modulo $C^\infty$, the delta distribution,
i.e.~we want
\begin{equation}
\label{distribution 2}
\int
\er^{\ir\, \phi(x;0,\eta)}
\,\mathfrak{s}(0,\eta)
\,\chi(x;0,\eta)
\,u(x;0,\eta)
\,k(x)\,\dr x\,\dbar \eta\,
=k(0)
\end{equation}
modulo a smooth functional.
Substituting \eqref{definition of u} into \eqref{distribution 2} we rewrite the latter as
\begin{equation}
\label{distribution 3}
\int
\er^{\ir\, \phi(x;0,\eta)}
\,\mathfrak{s}(0,\eta)
\,\chi(x;0,\eta)
\,\kappa(x)\,\sqrt{\det\phi_{x\eta}}
\ \dr x\,\dbar \eta\,
=\kappa(0)\,,
\end{equation}
where $\kappa(x)=\rho(x)^{1/2}\,k(x)$
and the branch of the square root is chosen so that $\sqrt{\det\phi_{x\eta}}=1$ at $x=0$;
see also Remark~\ref{remark on fourth root}.
Formula \eqref{distribution 3} is, in turn, equivalent to
\begin{equation}
\label{distribution 4}
\int
\er^{\ir\, \phi(x;0,\eta)}
\,\mathfrak{s}(0,\eta)
\,\chi(x;0,\eta)
\,\sqrt{\det\phi_{x\eta}}
\ \dbar \eta\,
=
\int
\er^{\ir\,x^\alpha\eta_\alpha}
\,\dbar \eta\,.
\end{equation}
The integrals in  \eqref{distribution 4} are understood as distributions in the variable $x$
and equality is understood as equality modulo a smooth distribution.

The complex exponential in
\eqref{distribution 4}
admits the expansion
\begin{equation}\label{identity FIO temp 1}
\er^{\ir \,\phi(x;0,\eta)}=\er^{\ir\,x^\alpha\eta_\alpha}\,\left[1 + \frac{\ir}{2}\, \phi_{x^\mu x^\nu}(0;0,\eta)\,x^\mu\,x^\nu +O(\|x\|^3) \right]\,.
\end{equation}
Furthermore,
\begin{equation}\label{identity FIO temp 2}
\sqrt{\det\phi_{x\eta}}\,(x;0,\eta)= 1+ \frac12\,\phi_{x^\alpha x^\beta\eta_\alpha}(0;0,\eta)\,x^\beta +O(\|x\|^2).
\end{equation}
Substituting \eqref{identity FIO temp 1} and \eqref{identity FIO temp 2} into
the LHS of \eqref{distribution 4}
and integrating by parts we get
\begin{multline*}
\int \er^{\ir\,x^\gamma\eta_\gamma}\,\Bigl( 1+\mathfrak{s}_{-1}(0,\eta)
-
\frac{\ir}{2}\, \phi_{x^\mu x^\nu \eta_\mu \eta_\nu}(0;0,\eta)
\\
+
\frac{\ir}2\,\phi_{x^\alpha x^\beta\eta_\alpha\eta_\beta}(0;0,\eta)
\,+O(\|\eta\|^{-2})\Bigr)\,\dbar\eta\\
\qquad=\int \er^{\ir\,x^\gamma\eta_\gamma}\,\bigl( 1+ \mathfrak{s}_{-1}(0,\eta)+O(\|\eta\|^{-2})\bigr)\,\dbar\eta\,,
\end{multline*}
from which we conclude that $\mathfrak{s}_{-1}(0,\eta)=0$.

We have shown that $\mathfrak{s}_{-1}$ vanishes identically
on the punctured
cotangent fibre at the point $P\in M$.
As the point $P$ is arbitrary
and $\mathfrak{s}_{-1}$ is a scalar function,
we conclude that
$\,\mathfrak{s}_{-1}(y,\eta)=0$, $\forall(y,\eta)\in T'M$.
\end{proof}

Stronger results can be established for the Levi-Civita phase function.

\begin{theorem}
\label{symbol identity Levi-Civita epsilon not 0}
The sub-subleading contribution to
the symbol of the identity operator written as a pseudo\-differential operator \eqref{identity as FIO} with the Levi-Civita phase function
$\varphi(0,x;y,\eta;\epsilon)$ is
\begin{equation}\label{subsubsymbol identity for Levi-Civita}
\mathfrak{s}_{-2}(y,\eta)=
\frac{(d-1)\,(d-2)\,\epsilon^2}{8\,g^{\alpha\beta}(y)\,\eta_\alpha\eta_\beta}\,.
\end{equation}
\end{theorem}

\begin{proof}
Let us fix a point $P\in M$ and argue as in the proof of
Theorem~\ref{theorem subprincipal identity FIO}, arriving at \eqref{distribution 4}.
Note that in this argument we did not specify the choice of a local coordinate system
in a neighbourhood of the point $P$.

Let us choose geodesic normal coordinates centred at $P$.
Then the explicit formula for the phase function appearing in \eqref{distribution 4} reads
\begin{equation*}
\label{symbol identity Levi-Civita epsilon not 0 proof formula 1}
\phi(x;0,\eta)=\varphi(0,x;0,\eta;\epsilon)
=x^\alpha\eta_\alpha
+\dfrac{\ir\,\epsilon}{2}\,\|\eta\|\,\|x\|^2\,,
\end{equation*}
where $\|\,\cdot\,\|$ stands for the Euclidean norm, see also \eqref{phase function in normal coordinates}.

The complex exponential in
\eqref{distribution 4}
admits the expansion
\begin{equation}
\label{symbol identity Levi-Civita epsilon not 0 proof formula 2}
\er^{\ir\,\phi(x;0,\eta)}=\er^{\ir\,x^\alpha\eta_\alpha}\,
\left[
1
-
\dfrac{\epsilon}{2}\,\|\eta\|\,\|x\|^2
+
\dfrac{\epsilon^2}{8}\,\|\eta\|^2\,\|x\|^4
+
O(\|x\|^6)
\right]\,.
\end{equation}

We have
\begin{equation}
\label{symbol identity Levi-Civita epsilon not 0 proof formula 3}
(\phi_{x^\alpha\eta_\beta})(x;0,\eta)
=
\delta_\alpha{}^\beta
+
\ir\,\epsilon\,\delta_{\alpha\mu}\,\delta^{\beta\nu}
\dfrac{\eta_\nu}{\|\eta\|}
\,x^\mu\,.
\end{equation}
It is well know that,
given the identity matrix $I$ and arbitrary small square matrix $A$ of the same size,
the expansion for $\det(I+A)$ reads
\begin{equation}
\label{symbol identity Levi-Civita epsilon not 0 proof formula 4}
\det(I+A)
=
1
+
\operatorname{tr}A
+
\frac12
\bigl[
(\operatorname{tr}A)^2
-
\operatorname{tr}(A^2)
\bigr]
+
O(\|A\|^3).
\end{equation}
Formulae
\eqref{symbol identity Levi-Civita epsilon not 0 proof formula 3}
and
\eqref{symbol identity Levi-Civita epsilon not 0 proof formula 4}
imply
\begin{equation}
\label{symbol identity Levi-Civita epsilon not 0 proof formula 5}
\sqrt{\det\phi_{x\eta}}\,(x;0,\eta)=
1
+
\dfrac{\ir\epsilon}{2\|\eta\|}\,x^\alpha\eta_\alpha
+
\dfrac{\epsilon^2}{8\|\eta\|^2}\,(x^\beta\eta_\beta)^2
+O(\|x\|^3).
\end{equation}
Substituting
\eqref{symbol identity Levi-Civita epsilon not 0 proof formula 2}
and
\eqref{symbol identity Levi-Civita epsilon not 0 proof formula 5} into
the LHS of \eqref{distribution 4}
and integrating by parts we get
\begin{multline*}
\label{symbol identity Levi-Civita epsilon not 0 proof formula 6}
\int \er^{\ir\,x^\gamma\eta_\gamma}\,\Bigl( 1
+\mathfrak{s}_{-2}(0,\eta)
-
\frac{\epsilon^2}{8}
\left(
\frac{\eta_\alpha\eta_\beta}{\|\eta\|^2}
\right)_{\eta_\alpha\eta_\beta}
+O(\|\eta\|^{-3})\Bigr)\dbar\eta
\\
=\int \er^{\ir\,x^\gamma\eta_\gamma}\,
\left( 1+ \mathfrak{s}_{-2}(0,\eta)
-
\frac{(d-1)\,(d-2)\,\epsilon^2}{8\|\eta\|^2}
+O(\|\eta\|^{-3})\right)\dbar\eta\,,
\end{multline*}
which gives us \eqref{subsubsymbol identity for Levi-Civita}.
\end{proof}

The algorithm described in the proof of
Theorem~\ref{symbol identity Levi-Civita epsilon not 0}
allows one to calculate explicitly $\mathfrak{s}_{-3},\mathfrak{s}_{-4},\ldots$ but the calculations
become cumbersome. We list the resulting formulae for the special case $d=2$:
\begin{equation}
\begin{aligned}
&\mathfrak{s}_{-3}(y,\eta)=\dfrac{1}{2^3}\,\dfrac{\epsilon^3}{(g^{\alpha\beta}(y)\,\eta_\alpha\eta_\beta)^{3/2}}\,,
\qquad
&\mathfrak{s}_{-4}(y,\eta)=0\,,
\\
&\mathfrak{s}_{-5}(y,\eta)=
\dfrac{3^2\times 5}{2^6}\,
\dfrac{\epsilon^5}{(g^{\alpha\beta}(y)\,\eta_\alpha\eta_\beta)^{5/2}}\,,
\qquad
&\mathfrak{s}_{-6}(y,\eta)=0\,,
\\
&\mathfrak{s}_{-7}(y,\eta)=
\dfrac{3^2\times 5^2\times\,13}{2^{10}}\,
\dfrac{\epsilon^7}{(g^{\alpha\beta}(y)\,\eta_\alpha\eta_\beta)^{7/2}}\,,
\qquad
&\mathfrak{s}_{-8}(y,\eta)=0\,,
\\
&\mathfrak{s}_{-9}(y,\eta)=
\dfrac{3^3\times 5^2\times 7^2\times 47}{2^6}\,
\dfrac{\epsilon^9}{(g^{\alpha\beta}(y)\,\eta_\alpha\eta_\beta)^{9/2}}\,,
\qquad
&\mathfrak{s}_{-10}(y,\eta)=0\,.
\end{aligned}
\end{equation}

We have an
even stronger result for the real-valued Levi-Civita phase function.
The following theorem holds for Riemannian manifolds $M$ of arbitrary dimension $d$.

\begin{lemma}\label{symbol identity Levi-Civita epsilon=0}
The full symbol of the identity operator written as a pseudodifferential operator \eqref{identity as FIO} with the real-valued Levi-Civita phase function $\varphi(0,x;y,\eta;0)$ is
\begin{equation}\label{symbol identity for real Levi-Civita}
\mathfrak{s}(y,\eta)=1.
\end{equation}
\end{lemma}

\begin{proof}
Formula \eqref{symbol identity for real Levi-Civita} is established by arguing as in the proof of
Theorem~\ref{symbol identity Levi-Civita epsilon not 0}.
\end{proof}

\section{The $g$-subprincipal symbol of the propagator}
\label{The subprincipal symbol of the propagator}

Sometimes, for particular purposes (e.g.~in spectral theory), one needs
only a few leading homogeneous components of
the full symbol $\mathfrak{a}$. In this section we will revisit and analyse further the construction of Section \ref{The global invariant symbol of the propagator} for the special case of the $g$-subprincipal symbol.

\begin{definition}
We call the scalar function $\mathfrak{a}_{-1}(t;y,\eta;\epsilon)$ appearing in
Definition~\ref{invariant definition of symbol} the \emph{$g$-subprincipal symbol of the wave propagator}.
\end{definition}

Acting with the wave operator
\eqref{definition of the wave operator}
on the oscillatory integral
\begin{equation}\label{oscillatory integral subprincipal}
\int \er^{\ir\,\varphi(t,x;y,\eta;\epsilon)}\, \left( 1+\mathfrak{a}_{-1}(t;y,\eta;\epsilon) \right)\, w(t,x;y,\eta;\epsilon)\, \dbar\eta,
\end{equation}
one obtains a new oscillatory integral 
\begin{equation*}
\int \er^{\ir\,\varphi(t,x;y,\eta;\epsilon)}\,a(t,x;y,\eta;\epsilon)\, w(t,x;y,\eta;\epsilon)\, \dbar\eta,
\end{equation*}
with
\begin{equation}\label{a for subprincipal}
a=(1+\mathfrak{a}_{-1})\,\er^{-\ir\,\varphi} \,\left[\mathcal{P}\,(\er^{\ir\,\varphi}\,w) \right] w^{-1} + (\mathfrak{a}_{-1})_{tt} + 2\,(\mathfrak{a}_{-1})_{t} \left(\ir\, \varphi_t + w_t \,w^{-1}  \right)\,.
\end{equation}
Here and in the following we drop the arguments for the sake of clarity.

\begin{lemma}
\label{lemma decomposition of b}
The function
\[
b(t,x;y,\eta;\epsilon):=\er^{-\ir\,\varphi} \,\left[\mathcal{P}\,(\er^{\ir\,\varphi}\,w) \right] w^{-1} 
\]
decomposes as
$b=b_2+b_1+b_0$, where 
\begin{subequations}
\begin{equation}\label{b2}
b_2= -\,(\varphi_t)^2 + 
\|\nabla \varphi\|_g^2,
\end{equation}
\begin{equation}\label{b1}
\begin{split}
b_1
&
=
\ir\,\left[ \varphi_{tt} - \Delta \varphi +2\, (\ln \,w)_t\,\varphi_t- 2 \,\langle \nabla(\ln\, w) , \nabla \varphi \rangle \right]
\end{split}
\end{equation}
\begin{equation}\label{b0}
b_0= w^{-1}  \left[ w_{tt} - \Delta w\right],
\end{equation}
\end{subequations}
the $b_{k}$, $k=2,1,0$, are positively homogeneous in $\eta$ of degree $k$, and
$\nabla$ is the Levi-Civita connection acting in the variable $x$.
\end{lemma}

\begin{proof}
The contribution to $b$ from the Laplacian reads
\begin{equation}\label{proof components of b eq 1}
\begin{split}
-\er^{-\ir\,\varphi}\,\Delta\left( \er^{\ir \,\varphi}\,w \right)  w^{-1}
&=
-\ir\,\Delta \varphi + \langle\nabla \varphi , \nabla \varphi\rangle - w^{-1}\Delta w -2\ir\,\langle \nabla \varphi, \nabla w \rangle w^{-1}.
\end{split}
\end{equation}
On the other hand, the contribution from the second derivative in time is
\begin{equation}\label{proof components of b eq 2}
\begin{split}
\er^{-\ir\,\varphi}\,\dfrac{\partial^2 }{\partial t^2}\left( \er^{\ir \,\varphi}\,w \right) w^{-1}= -\,(\varphi_t)^2 + \ir\,\varphi_{tt} + 2\,\ir\,\varphi_t\,w_t\,w^{-1} + w_{tt}\,w^{-1}\,.
\end{split}
\end{equation}
Combining \eqref{proof components of b eq 1} and \eqref{proof components of b eq 2}, and singling out terms with the same degree of homogeneity we arrive at \eqref{b2}--\eqref{b0}.
\end{proof}

In terms of the homogeneous components of $b$, formula \eqref{a for subprincipal} reads
\begin{equation}\label{a for subprincipal with bs}
\begin{split}
a= &\,b_2\\
+&\, b_1 + b_2\,\mathfrak{a}_{-1}\\
+&\, b_0+ b_1\,\mathfrak{a}_{-1}+ 2\,\ir\,( \mathfrak{a}_{-1})_t\,\varphi_t\\
+&\, b_0\,\mathfrak{a}_{-1}+( \mathfrak{a}_{-1})_{tt} + 2\,\mathfrak{a}_{-1}\,w_t\, w^{-1}\,,
\end{split}
\end{equation}
where we arranged on different lines contributions of decreasing degree of homogeneity, from $2$ to~$-1$.

Before constructing the amplitude-to-symbol operator and writing down the transport equations, we need a few preparatory lemmata.

\begin{lemma}
\label{lemma PFt along flow}
We have
\begin{equation}\label{PFt along flow}
\left.\varphi_t\right|_{x=x^*}=-h(y,\eta).
\end{equation}
\end{lemma}

\begin{proof}
Differentiating in $t$ both sides of (i)
in Definition \ref{phase function of class Fh}, one obtains
\begin{equation*}
\begin{split}
0=& \,\left.\varphi_t\right|_{x=x^*}+\left.\varphi_{x^\alpha}\right|_{x=x^*}\,\dot x^*{}^\alpha\\
=&\, \left.\varphi_t\right|_{x=x^*}+\xi^*_\alpha\,h_{\xi_\alpha}(x^*,\xi^*)\\
=& \, \left.\varphi_t\right|_{x=x^*}+h(x^*,\xi^*)\,.
\end{split}
\end{equation*}
In the second step condition (ii) from Definition \ref{phase function of class Fh} has been used, whereas the last step is a consequence of Euler's theorem on homogeneous functions. Formula \eqref{PFt along flow} now follows from the fact that the Hamiltonian is preserved along the flow.
\end{proof}

\begin{lemma}\label{lemma b_2 second order zero}
The function $b_2$ defined by \eqref{b2} has a second order zero in $x$ at $x=x^*(t;y,\eta)$, namely,
\begin{gather*}
\left.b_2\right|_{x=x^*}=0, \qquad
\left.\nabla b_2\right|_{x=x^*}=0.
\end{gather*}
\end{lemma}
\begin{proof}
Rewriting $b_2$ as
\begin{equation*}
b_2=-(\varphi_t)^2+ h^2(x,\nabla \varphi),
\end{equation*}
one immediately concludes that $b_2$ vanishes along the flow by Lemma~\ref{lemma PFt along flow} and Definition~\ref{phase function of class Fh}, condition~(ii). 

Proving that the derivative vanishes as well is slightly trickier. We have
\begin{equation*}
\begin{split}
\nabla_\mu b_2 = -2 \,\varphi_t\,\varphi_{tx^\mu} + 2\, h(x,\nabla \varphi) \,[h(x,\nabla \varphi)]_{x^\mu},
\end{split}
\end{equation*}
from which it ensues, by evaluating along the flow, that
\begin{equation*}
\begin{split}
\left.\nabla_\mu b_2\right|_{x=x^*}=&
\left. \bigl(-2 \,\varphi_t\,\varphi_{tx^\mu} + 2\, h(x,\nabla \varphi) \,[h(x,\nabla \varphi)]_{x^\mu}\bigr)  \right|_{x=x^*}\\
=& \,2\,  h(y,\eta) \left. \bigl(\varphi_{tx^\mu} + [h(x,\nabla \varphi)]_{x^\mu} \bigr)  \right|_{x=x^*}  ,
\end{split}
\end{equation*}
where, once again, we used Lemma \ref{lemma PFt along flow}. The problem at hand is now down to showing that 
\begin{equation}
\label{second order zero temp1}
\left. \bigl(\varphi_{tx^\mu} + [h(x,\nabla \varphi)]_{x^\mu} \bigr)  \right|_{x=x^*} =0.
\end{equation}
From the general properties of a phase function of class $\mathcal{L}_h$, one argues that, in an arbitrary coordinate system, $\varphi$ can be represented as
\begin{equation}\label{expansion of general PF in x}
\varphi=(x-x^*)^\alpha\,\xi^*_\alpha+ \frac12\, [H_\varphi]_{\mu\nu}\,(x-x^*)^\mu\,(x-x^*)^\nu
+O(\|x-x^* \|^3),
\end{equation}
with
\[
[H_\varphi]_{\alpha\beta}:=
\left.\varphi_{x^\alpha x^\beta}\right|_{x=x^*}.
\]
Combining \eqref{expansion of general PF in x} with Hamilton's equations, we get
\begin{equation}
\label{second order zero temp2}
\left.\varphi_{tx^\alpha}\right|_{x-x^*}= \dot \xi^*_\alpha - [H_\varphi]_{\alpha\mu}\,\dot x^*{}^\mu = -h_{x^\alpha}(x^*,\xi^*) - [H_\varphi]_{\alpha\mu}\,h_{\xi_\mu}(x^*,\xi^*).
\end{equation}
Moreover, we have
\begin{equation}
\begin{split}
\label{second order zero temp3}
\left. [h(x,\nabla \varphi)]_{x^\alpha} \right|_{x=x^*}=&h_{x^\alpha}(x^*,\xi^*) + h_{\xi_\mu}(x^*,\xi^*) \,\left. \varphi_{x^\alpha x^\mu}\right|_{x=x^*}\\
=&h_{x^\alpha}(x^*,\xi^*) + [H_\varphi]_{\alpha\mu}\,h_{\xi_\mu}(x^*,\xi^*).
\end{split}
\end{equation}
Substitution of \eqref{second order zero temp2} and \eqref{second order zero temp3} into \eqref{second order zero temp1} concludes the proof.
\end{proof}

Lemmata
\ref{lemma PFt along flow} and \ref{lemma b_2 second order zero}
are not specific to the Levi-Civita
phase function: they remain true for any
phase function of the class $\mathcal{L}_h$.

We are now in a position to analyse the transport equations. With the notation from Section~\ref{The global invariant symbol of the propagator}, in view of formulae \eqref{construction homogeneous components of the reduced amplitude} and
\eqref{a for subprincipal with bs}  we have
\begin{subequations}
\begin{align}
\mathfrak{b}_2=& \mathfrak{S}_0\,b_2\,,\label{mathfrak b2}\\
\mathfrak{b}_1=& \mathfrak{S}_{-1}\,b_2 + \mathfrak{S}_0\,b_1\,,\label{mathfrak b1}\\
\mathfrak{b}_0=& \mathfrak{S}_{-2}\,b_2 + \mathfrak{S}_{-1}\,b_1+\mathfrak{S}_{0}\, b_0  - 2\,\ir\,h\,(\mathfrak{a}_{-1})_t + \mathfrak{a}_{-1}\,\mathfrak{b}_1\,.\label{mathfrak b0}
\end{align}
\end{subequations}
Note that homogeneous components of the symbol $\mathfrak{a}_{-k}$ with degree of homogeneity less than $-1$, even if taken into account in \eqref{oscillatory integral subprincipal}, would not contribute to \eqref{mathfrak b2}--\eqref{mathfrak b0}.
Note also the appearance of the $x$-independent term
$\mathfrak{a}_{-1}\,\mathfrak{b}_1$ on the RHS
of \eqref{mathfrak b0}: it can be traced back to the fact that Lemma~\ref{lemma b_2 second order zero} implies
\[
\mathfrak{S}_{-1}(b_2\,\mathfrak{a}_{-1})
=
\mathfrak{a}_{-1}\mathfrak{S}_{-1}b_2.
\]

The zeroth transport equation $\mathfrak{b}_2=0$
is clearly satisfied, due to Lemma~\ref{lemma b_2 second order zero}.

\begin{lemma}\label{lemma FTE}
The first transport equation (FTE) $\mathfrak{b}_1=0$ can be equivalently rewritten as
\begin{equation}\label{FTE alternative form}
\left.\left(\varphi_{tt} - \Delta\varphi  \right)\right|_{x=x^*} = 2\,h \,\dfrac{\dr(\ln\, w^*)}{\dr t}+\frac{1}{2}\,(x^*_{\eta_\alpha})^\gamma\left.\left[ [(\varphi_{x\eta})^{-1}]_{\alpha}{}^\beta\,(b_2)_{x^\beta x^\gamma}  \right] \right|_{x=x^*}\,,
\end{equation}
where
$w^*(t;y,\eta;\epsilon)=w(t,x^*(t;y,\eta);y,\eta;\epsilon)$.
\end{lemma}
\begin{proof}
Consider the operator $\mathfrak{S}_{-1}$ defined in \eqref{mathfrak Sk}. When acting on a function with a second order zero along the flow, it can be simplified to read
\begin{equation}\label{mathfrak S-1 temp}
\mathfrak{S}_{-1}\,b_2=\left.\ir\,\dfrac{\partial(L_\beta\, b_2)}{\partial \eta_\beta}\right|_{x=x^*} - \frac{\ir}{2}\left.\left[ \varphi_{\eta_\alpha \eta_\beta}\,L_\alpha L_\beta\,b_2  \right]\right|_{x=x^*}.
\end{equation}
Here we used the fact that $\mathfrak{S}_0 \,\varphi_\eta=0$. Using the notation
$
H_f:=\left.f_{xx}\right|_{x=x^*}
$
and putting
$
\Phi_{x\eta}:=\left.\varphi_{x\eta} \right|_{x=x^*}\,
$,
we observe that
\[
(\Phi_{x\eta})_\alpha{}^\beta= (\xi^*_{\eta_\beta})_\alpha- (H_\varphi)_{\alpha\mu}\,(x^*_{\eta_\beta})^\mu
\]
and, consequently,
\begin{equation*}
\begin{split}
\left.\varphi_{\eta_\alpha\eta_\beta}\right|_{x=x^*}=& -(x^\ast_{\eta_\alpha})^\gamma\, (\xi^\ast_{\eta_\beta})_\gamma +  (H_\varphi)_{\mu\nu}\, (x^\ast_{\eta_\alpha})^\mu (x^\ast_{\eta_\beta})^\nu\\
&= - (x^\ast_{\eta_\alpha})^\gamma \left[ (\xi^\ast_{\eta_\beta})_\gamma -  (H_\varphi)_{\gamma\nu}\, (x^\ast_{\eta_\beta})^\nu \right]\\
&=  - (x^\ast_{\eta_\alpha})^\gamma \,(\Phi_{x\eta})_\gamma{}^\beta.
\end{split}
\end{equation*}
Hence, recalling formula \eqref{operator L}, we obtain
\begin{equation}\label{mathfrak S1 sub step 1}
\begin{split}
- \frac{\ir}{2}\left.\left[ \varphi_{\eta_\alpha \eta_\beta}\,L_\alpha L_\beta\,b_2  \right]\right|_{x=x^*}=& \dfrac{\ir}{2} (x^\ast_{\eta_\alpha})^\gamma \, (\Phi_{x\eta})_\gamma{}^\beta \, (\Phi^{-1}_{x\eta})_\alpha{}^\delta \, (\Phi^{-1}_{x\eta})_\beta{}^\rho \, (H_{b_2})_{\delta\rho}\\
=&  \dfrac{\ir}{2} (x^\ast_{\eta_\alpha})^\gamma \, \delta_\gamma{}^\rho \, (\Phi^{-1}_{x\eta})_\alpha{}^\delta \, (H_{b_2})_{\delta\rho}\\
=&  \dfrac{\ir}{2} (x^\ast_{\eta_\alpha})^\rho \,  (\Phi^{-1}_{x\eta})_\alpha{}^\delta \, (H_{b_2})_{\delta\rho}\, .
\end{split}
\end{equation}
Furthermore, upon writing
\[
b_2=\frac12(H_{b_2})_{\alpha\beta}\,(x-x^*)^\alpha\,(x-x^*)^\beta+ O(\|x-x^* \|^3),
\]
the first term in \eqref{mathfrak S-1 temp} becomes
\begin{equation}\label{mathfrak S1 sub step 2}
\left.\ir\,\dfrac{\partial(L_\beta\, b_2)}{\partial \eta_\beta}\right|_{x=x^*}=  - \ir\,  (x^\ast_{\eta_\alpha})^\gamma \,  (\Phi^{-1}_{x\eta})_\alpha{}^\mu \, (H_{b_2})_{\mu\gamma}\,.
\end{equation}
By substituting \eqref{mathfrak S1 sub step 1} and \eqref{mathfrak S1 sub step 2} into \eqref{mathfrak S-1 temp} we arrive at the last summand in the  RHS of \eqref{FTE alternative form}. As for the remaining terms, they correspond to $\mathfrak{S}_0 \,b_1$ in \eqref{mathfrak b1} and are obtained by evaluating \eqref{b1} along the flow and performing straightforward algebraic manipulations.
\end{proof}

It is possible to show directly, by means of a long and tedious, though non-trivial, computation that \eqref{FTE alternative form} is satisfied automatically, thus providing a direct proof that the principal symbol of the wave propagator is indeed $1$. If one started with a generic term $\mathfrak{a}_0$ in \eqref{oscillatory integral subprincipal}, the FTE would be an ordinary differential equation allowing for the (unique) determination thereof. Lemma~\ref{lemma FTE} gives us an explicit formula for the action of the wave operator on the Levi-Civita phase function.

Let us now move on to the second transport equation $\mathfrak{b}_0=0$, the one that allows for the determination of the $g$-subprincipal symbol $\mathfrak{a}_{-1}(t;y,\eta;\epsilon)$. For computing the $g$-subprincipal symbol, a simplified representation of the operators
$\mathfrak{S}_{-1}$ and $\mathfrak{S}_{-2}$ may be used.
Recall that for general $k$ the operators
$\mathfrak{S}_{-k}$ are defined by formulae \eqref{mathfrak S0}
and \eqref{mathfrak Sk}. Put
\begin{equation}\label{mathfrak B_-1}
\mathfrak{B}_{-1}:=\ir\, w^{-1}\,\dfrac{\partial}{\partial \eta_\alpha}\,w\,L_\alpha -\dfrac{\ir}{2}\,\varphi_{\eta_\alpha \eta_\beta}\,L_\alpha\,L_\beta\,.
\end{equation}
Then we have
\begin{gather}
{\mathfrak{S}}_{-1}=\mathfrak{S}_0 \,\mathfrak{B}_{-1}\,, \label{mathfrak S tilde -1}\\
{\mathfrak{S}}_{-2}=\mathfrak{S}_0 \,\mathfrak{B}_{-1} \left[ \ir \, w^{-1} \frac{\partial}{\partial \eta_\beta}\, w \left( 1 +\sum_{1\leq |\boldsymbol{\alpha}|\leq 3} \dfrac{(-\varphi_\eta)^{\boldsymbol{\alpha}}}{\boldsymbol{\alpha}!(|\boldsymbol{\alpha}|+1)}\,L_{\boldsymbol{\alpha}} \right) L_\beta  \right],\label{mathfrak S tilde -2}
\end{gather}
and these representations can now be used in formula \eqref{mathfrak b0}. 

The last ingredient needed to write down the $g$-subprincipal symbol is the initial condition at $t=0$, extensively discussed in Section~\ref{Invariant representation of the identity operator}. The Levi-Civita phase function evaluated at $t=0$, $\varphi(0,x;y,\eta;\epsilon)$, clearly satisfies the assumptions (a) and (b) of Theorem~\ref{theorem subprincipal identity FIO}, hence $\left.\mathfrak{a}_{-1}\right|_{t=0}=0$. Integrating in time, we arrive at the following theorem.

\begin{theorem}\label{theorem ODE subprincipal symbol}
The global invariantly defined $g$-subprincipal symbol of the wave propagator is
\begin{equation}\label{ODE subprincipal symbol}
{\mathfrak{a}}_{-1}(t;y,\eta;\epsilon)=-\dfrac{\ir}{2\,h}
\int_0^t
\left[ {\mathfrak{S}}_{-2}\,b_2 +{\mathfrak{S}}_{-1}\,b_1 +\mathfrak{S}_0\,b_0   \right]
(\tau;y,\eta;\epsilon)\,\dr\tau\,.
\end{equation}
The functions $b_k$, $k=2,1,0$, are defined by \eqref{b2}--\eqref{b0}, \eqref{phase function with distance}, \eqref{weight w}, while the operators ${\mathfrak{S}}_{-2}$, ${\mathfrak{S}}_{-1}$ and $\mathfrak{S}_0$ are given by \eqref{mathfrak B_-1}--\eqref{mathfrak S tilde -2} and \eqref{operator L}, \eqref{mathfrak S0}.
\end{theorem}

\section{Small time expansion for the $g$-subprincipal symbol}
\label{Small time expansion for the subprincipal symbol}

The small time behaviour of the wave propagator carries important information about the spectral properties of the Laplace--Beltrami operator. Our geometric construction allows us to derive an explicit universal formula for the coefficient of the linear term in the expansion of the $g$-subprincipal symbol when $t$ tends to zero. In Appendix~\ref{Weyl coefficients} we will explain how this formula can be used to recover, in a straightforward manner, the third Weyl coefficient.

When time is sufficiently small we can use the real-valued Levi-Civita phase function, since condition (iii) in Definition~\ref{phase function of class Fh} is automatically satisfied.  Therefore, throughout this section we set $\epsilon=0$.

\begin{theorem}\label{theorem small time}
The $g$-subprincipal symbol of the wave propagator admits the following expansion for small times:
\begin{equation}\label{subprincipal small times}
\mathfrak{a}_{-1}(t;y,\eta)=\dfrac{\ir}{12\,h(y,\eta)}\,\mathcal{R}(y)\,t+ O(t^2)\,,
\end{equation}
where $\mathcal{R}$ is scalar curvature.
\end{theorem}

\begin{proof}
Let us fix an arbitrary point $y\in M$ and choose geodesic normal coordinates
centred at $y$. As $\,\mathfrak{a}_{-1}\,$ is a scalar function,
in order to prove the theorem it is sufficient to prove
\begin{equation}\label{subprincipal small times proof}
\mathfrak{a}_{-1}(t;0,\eta)=\dfrac{\ir}{12\,h(0,\eta)}\,\mathcal{R}(0)\,t+ O(t^2)
\end{equation}
in the chosen coordinate system.

As we are dealing with the case when $t$ tends to zero,
we can assume that $x^*$ and $x$ both lie in a geodesic neighbourhood of $y$.
In what follows we use for $x$ geodesic normal coordinates
centred at $y$ and perform a double Taylor expansion of the phase function in powers of $t$ and $x$ simultaneously. We shall also assume that $t$ and $\|x\|$ are of the same order  and write $O(\|x\|^n+|t|^n)$ as a shorthand for $O(\|x\|^p \,t^{n-p})$ for all $p\in\{0,1,\ldots,n\}$.

It is well known that in geodesic normal coordinates centred at $y$ we have
\begin{equation}
\label{expansion xstar small t}
x^*{}^\alpha=\frac{\eta^\alpha}{h}\,t\,,
\end{equation}
where $\eta^\alpha=\delta^{\alpha\beta}\eta_\beta$.
Substituting \eqref{expansion xstar small t}
into the first Hamilton's equation \eqref{Hamiltonian flow} we get
\begin{equation}
\label{expansion xistar small t}
\begin{split}
\xi^*_\alpha
&=
g_{\alpha\beta}
(x^*)\,\eta^\beta
\\
&=
\left[ g_{\alpha\beta}
\left(
\frac{\eta}{h}\, t
\right)\right]\eta^\beta
\\
&=
\left(
\delta_{\alpha\beta}
-\frac13\,R_{\alpha\mu\beta\nu}(0)\,\frac{\eta^\mu\eta^\nu}{h^2}\,t^2+O(t^3)
\right)\eta^\beta
\\
&=\eta_\alpha+O(t^3).
\end{split}
\end{equation}
The simplifications in the above calculations are due to the properties of normal coordinates and the (anti)symmetries of the Riemann curvature tensor $R$. In fact,  using the Gauss Lemma, one can show that the remainder in \eqref{expansion xistar small t} is zero, i.e.~$\xi^*_\alpha=\eta_\alpha$.

Arguing as in the proof of Theorem~\ref{theorem phixeta at x=xstar}
and using formula \eqref{expansion xstar small t},
one concludes that the initial velocity of the (unique) geodesic connecting $x^*$ to $x$ is
\begin{equation}\label{proof small time temp1}
\begin{split}
\dot \gamma (0)^\alpha&=(x-x^*)^\alpha+ \frac12 \,\Gamma^\alpha{}_{\beta\gamma}(x^*)\,(x-x^*)^\beta\,(x-x^*)^\gamma\\
& + \frac12\,(\partial_{x^\mu}\Gamma^\alpha{}_{\beta\gamma})(x^*)\,(x-x^*)^\mu\, (x-x^*)^\beta\, (x-x^*)^\gamma\\
&+ O(\|x-x^*\|^4)
\\
&=x^\alpha-\frac{\eta^\alpha}{h}\,t
+\frac12(\partial_{x^\mu}\Gamma^\alpha{}_{\beta\gamma})(0)\,\frac{\eta^\mu}{h} \, t\,
\left(x-\frac\eta h\,t\right)^\beta
\left(x-\frac\eta h\,t\right)^\gamma\\
&+\frac12(\partial_{x^\mu}\Gamma^\alpha{}_{\beta\gamma})(0)
\left(x-\frac\eta h\,t\right)^\mu
\left(x-\frac\eta h\,t\right)^\beta
\left(x-\frac\eta h\,t\right)^\gamma
+O(\|x\|^4+t^4)
\\
&=x^\alpha-\frac{\eta^\alpha}{h} \, t-(\partial_{x^\mu}\Gamma^\alpha{}_{\beta\gamma})(0)\,\frac{\eta^\beta}{h} \, t\, x^\mu\, x^\gamma +\frac12(\partial_{x^\mu}\Gamma^\alpha{}_{\beta\gamma})(0)\,\frac{\eta^\beta\,\eta^\gamma}{h^2} \, t^2\,x^\mu \\
&+\frac12(\partial_{x^\mu}\Gamma^\alpha{}_{\beta\gamma})(0)\,
x^\mu
x^\beta
x^\gamma
+O(\|x\|^4+t^4)
\\
&=x^\alpha-\frac{\eta^\alpha}{h} \, t+\frac1{3\,h} \,R^\alpha{}_{\gamma\beta\mu}(0)\,\eta^\beta \, t\,x^\gamma x^\mu -\frac13 R^\alpha{}_{\beta\gamma\mu}(0)\,\frac{\eta^\beta\,\eta^\gamma}{h^2} \, t^2\, x^\mu\\
&+O(\|x\|^4+t^4)\,.
\end{split}
\end{equation}
Here at the last step we resorted to the identity
\begin{equation}\label{identity gammas vs Riemann normal coordinates}
(\partial_{x^\mu}\Gamma^\alpha{}_{\beta\gamma})(0)=-\frac{1}{3}\,(R^\alpha{}_{\beta\gamma\mu}(0)+R^\alpha{}_{\gamma\beta\mu}(0)).
\end{equation}

Lemma~\ref{lemma about versions of Re PF} and
formulae \eqref{expansion xistar small t}, \eqref{proof small time temp1}
imply that our real-valued Levi-Civita
phase function admits the following Taylor
expansion in powers of $x$ and $t$:
\begin{equation}\label{phase function small time}
\varphi(t,x;0,\eta)=x^\alpha\,\eta_\alpha-t\,h
+\frac1{3\,h} \,R^\alpha{}_\mu{}^\beta{}_\nu(0)\,\eta_\alpha\,\eta_\beta \, t\,x^\mu\, x^\nu
+O(\|x\|^4+t^4)\,.
\end{equation}

The next step is computing the homogeneous functions $b_2$, $b_1$ and $b_0$ defined by \eqref{b2}--\eqref{b0}
at $t=0$.

Direct inspection tells us that 
\begin{equation*}
\begin{split}
\left.-(\varphi_t)^2\right|_{t=0}=-h^2
+\frac23 \,R^\alpha{}_\mu{}^\beta{}_\nu(0)\,\eta_\alpha\,\eta_\beta\,x^\mu\, x^\nu
+O(\|x\|^3)
\end{split}
\end{equation*}
and
\begin{equation*}
\begin{split}
\left. g^{\alpha\beta}(x)\,\varphi_{x^\alpha}\,\varphi_{x^\beta}\right|_{t=0}= h^2
 +\frac13 \,R^\alpha{}_\mu{}^\beta{}_\nu(0)\,\eta_\alpha\,\eta_\beta \,x^\mu\, x^\nu
+O(\| x \|^3)\,.
\end{split}
\end{equation*}
Adding up the above two formulae, we get
\begin{equation}
\label{dimas label for b2 at time zero}
b_2(0,x;0,\eta)=R^\alpha{}_\mu{}^\beta{}_\nu(0)\,\eta_\alpha\,\eta_\beta\,x^\mu\, x^\nu
+O(\| x  \|^3).
\end{equation}

Let us now move on to $b_1$. Direct differentiation of \eqref{phase function small time} reveals that
\begin{equation}
\label{b1 small time step 1}
\left.\varphi_{tt}\right|_{t=0}=O(\| x  \|^2), \qquad \left.\varphi_t\right|_{t=0}=-h+O(\|x\|^2)\,, \qquad \left.\varphi_{xx}\right|_{t=0}=O(\| x \|^2)\,.
\end{equation}
Furthermore, we have
\begin{equation}
\label{b1 small time step 2}
\varphi_{x^\rho\eta_\sigma}=\delta_\rho{}^\sigma+t\,\frac{2}{3h}
\left(
R^\sigma{}_\rho{}^\beta{}_\nu(0)\,\eta_\beta
+
R^\alpha{}_\rho{}^\sigma{}_\nu(0)\,\eta_\alpha
\right)
x^\nu
+O(\|x\|^3+|t|^3)
\end{equation}
and, consequently,
\begin{equation}
\label{b1 small time step 3}
\begin{split}
\det\varphi_{x^\rho\eta_\sigma}= 1-t\,\frac{2}{3h}\,
\operatorname{Ric}^\alpha{}_\nu(0)\,\eta_\alpha\,x^\nu
+O(\|x\|^3+|t|^3).
\end{split}
\end{equation}
Plugging
\eqref{b1 small time step 3} into \eqref{weight w} and expanding the Riemannian density in normal geodesic coordinates, one eventually obtains
\begin{equation}\label{w expansion for small t}
w=1+\frac1{12}\,\operatorname{Ric}_{\mu\nu}(0)\,x^\mu\,x^\nu -\frac t {3h}\,\operatorname{Ric}^\alpha{}_\nu(0)\,\eta_\alpha \,x^\nu
+O(\|x\|^3+|t|^3).
\end{equation}
Formulae
\eqref{identity gammas vs Riemann normal coordinates},
\eqref{phase function small time},
\eqref{b1 small time step 1}
and
\eqref{w expansion for small t}
give us
\begin{equation}\label{b1 small t temp1}
\begin{split}
&\left. - \ir  g^{\alpha\beta}(x)\,\nabla_\alpha \nabla_\beta\, \varphi\right|_{t=0} 
= - \ir 
g^{\alpha\beta}(x)\,
 \bigl(-\Gamma^\gamma{}_{\alpha\beta}(x)\left.\varphi_{x^\gamma}\right|_{t=0}+O(\| x  \|^2)\bigr)
\\
&\qquad= - \ir 
\bigl(\delta^{\alpha\beta}+O(\|x  \|^2)\bigr)
\bigl(
\tfrac13
(R^\gamma{}_{\alpha\beta\mu}(0)
+
R^\gamma{}_{\beta\alpha\mu}(0))
\,x^\mu
\,\eta_\gamma+O(\| x \|^2)
\bigr)
\\
&\qquad= 
\frac{2\,\ir}{3}\operatorname{Ric}^\gamma{}_\mu(0)\,\eta_\gamma\,x^\mu
+O(\| x \|^2),
\end{split}
\end{equation}
\begin{equation}\label{b1 small t temp2}
\left. 2\,\ir (\ln\, w)_t\,\varphi_t\right|_{t=0}=
\frac{2\,\ir}{3}\operatorname{Ric}^\alpha{}_\nu(0)\,\eta_\alpha\,x^\nu
+O(\| x  \|^2),
\end{equation}
\begin{equation}\label{b1 small t temp3}
\left. - 2\,\ir \,g^{\alpha\beta}(x)\,[\nabla_\alpha (\ln\, w)] \,\nabla_\beta\, \varphi\right|_{t=0}=
-\frac{\ir}{3}\operatorname{Ric}^\alpha{}_\nu(0)\,\eta_\alpha\,x^\nu
+O(\| x  \|^2).
\end{equation}
Substitution of
\eqref{b1 small time step 1}
and
\eqref{b1 small t temp1}--\eqref{b1 small t temp3}
into \eqref{b1}
yields
\begin{equation}
\label{dimas label for b1 at time zero}
b_1(0,x;0,\eta)= \ir\,\operatorname{Ric}^\alpha{}_\mu(0)\,\eta_\alpha\,x^\mu+O(\| x  \|^2).
\end{equation}

Finally, let us deal with $b_0$. Formula \eqref{w expansion for small t} implies that
\begin{equation*}
\begin{aligned}
&w=\,1+O(\|x\|^2+t^2),\qquad
&&w_{tt}=\, O(\|x\|+|t|),
\\
&w_x=\,O(\|x\|+|t|),\qquad
&&w_{xx}=\,\frac1{6}\,\operatorname{Ric}(0)+O(\|x\|+|t|).
\end{aligned}
\end{equation*}
Substituting the above formulae into  \eqref{b0}, we get
\begin{equation}
\label{dimas label for b0 at time zero}
b_0(0,x;0,\eta)=-\frac16\,\mathcal{R}(0)+O(\|x \|).
\end{equation}

Theorem \ref{theorem ODE subprincipal symbol} tells us that
\begin{equation}\label{dima small time temp 1}
\mathfrak{a}_{-1}(t;0,\eta)
=
-\dfrac{\ir}{2\,h}
\left.
\left[ {\mathfrak{S}}_{-2}\,b_2 +{\mathfrak{S}}_{-1}\,b_1 +\mathfrak{S}_0\,b_0 \right]
\right|_{t=0}t+ O(t^2)\,.
\end{equation}
Recall that the ${\mathfrak{S}}_{-2}$, ${\mathfrak{S}}_{-1}$ and $\mathfrak{S}_0$
in the above formula are the amplitude-to-symbol operators.

Calculating the last term in the square brackets in \eqref{dima small time temp 1} is easy.
Namely, using \eqref{dimas label for b0 at time zero}, we get
\begin{equation}\label{dima small time temp 2}
\left.
\left[\mathfrak{S}_0\,b_0 \right]
\right|_{t=0}
=
b_0(0,0;0,\eta)
=
-\frac16\,\mathcal{R}(0)\,.
\end{equation}

Calculating the first two terms in the square brackets in
\eqref{dima small time temp 1}
seems to be a challenging task because the formulae for the operators
${\mathfrak{S}}_{-2}$ and ${\mathfrak{S}}_{-1}$ are complicated.
However, at $t=0$ and in chosen local coordinates our phase function reads
\begin{equation*}\label{dima small time temp 3}
\varphi(0,x;0,\eta)=x^\alpha\,\eta_\alpha
\end{equation*}
and this leads to fundamental simplifications.
Namely, at $t=0$ we have
\begin{eqnarray}
\label{dima small time temp 4}
\left.
\left[\mathfrak{S}_{-1}\,(\ \cdot\ )\right]
\right|_{t=0}
&=&
\left.
\left[
\ir \,\dfrac{\partial}{\partial \eta_\alpha}\dfrac{\partial}{\partial x^\alpha}
\,(\ \cdot\ )
\right]
\right|_{t=0,\ x=0}
\,,\\
\label{dima small time temp 5}
\left.
\left[\mathfrak{S}_{-2}\,(\ \cdot\ )\right]
\right|_{t=0}
&=&
\dfrac12
\left.
\left[\left(
\ir \,\dfrac{\partial}{\partial \eta_\alpha}\dfrac{\partial}{\partial x^\alpha}\right)^2
(\ \cdot\ )
\right]
\right|_{t=0,\ x=0}
\,.
\end{eqnarray}

Substituting
\eqref{dimas label for b2 at time zero}
and
\eqref{dimas label for b1 at time zero}
into
\eqref{dima small time temp 4}
and
\eqref{dima small time temp 5}
respectively, we get
\begin{equation}\label{dima small time temp 6}
\left.
\left[\mathfrak{S}_{-2}\,b_2 \right]
\right|_{t=0}
=
\mathcal{R}(0)\,,
\qquad
\left.
\left[\mathfrak{S}_{-1}\,b_1 \right]
\right|_{t=0}
=
-\mathcal{R}(0)\,.
\end{equation}
Formulae
\eqref{dima small time temp 1},
\eqref{dima small time temp 2}
and
\eqref{dima small time temp 6}
imply
\eqref{subprincipal small times proof}.
\end{proof}

\section{Explicit examples}
\label{Explicit examples}
In this section we will apply our construction to the detailed analysis of two explicit examples.

\subsection{The 2-sphere}
The first example we will discuss is the 2-sphere. Clearly, for the 2-sphere one can construct the propagator via functional calculus, since eigenvalues and eigenfunctions are known explicitly. However, the 2-sphere is interesting as it represents, in a sense, the `most singular' instance of a Riemannian manifold in terms of obstructions caused by caustics because the geodesic flow on the cosphere bundle is $2\pi$-periodic.
Furthermore, geodesics focus at $t=\pi k$, $k\in\mathbb{Z}$.
As we will show, even in this simple example our method provides significant insight.

Let $\mathbb{S}^2$ be the standard 2-sphere embedded in Euclidean space $(\mathbb{E}^3, \delta_E:=\dr x^2+\dr y^2+\dr z^2)$ via the map $\iota:\mathbb{S}^2\to \mathbb{E}^3$, in such a way that the south pole is tangent to the plane $z=0$ at the origin $O=(0,0,0)$. The sphere is endowed with the standard round metric $g:=\iota^* \delta_E\,$.

Let us introduce coordinates on $\mathbb{S}^2$ minus the north pole by a stereographic projection onto the $xy$-plane,
\begin{gather}
\label{map sigma}
\sigma: \mathbb{R}^2 \to \mathbb{S}^2\setminus \begin{pmatrix}
0\\0\\2
\end{pmatrix}\,, \qquad
\begin{pmatrix}
u\\v
\end{pmatrix} \mapsto \begin{pmatrix}
x\\y\\z
\end{pmatrix} =\dfrac{1}{1+K^2}
\begin{pmatrix}
u\\
v\\
2 K^2
\end{pmatrix},
\end{gather}
where
$
K:=\frac{\sqrt{u^2+v^2}}{2}.
$
The metric in stereographic coordinates reads
\begin{equation}\label{metric sphere stereographic coordinates}
g=\dfrac{1}{(1+K^2)^2}\left[ \dr u^2 +\dr v^2 \right].
\end{equation}
Without loss of generality, we will set $y=(0,0)\in \mathbb{R}^2$ in stereographic coordinates. Further on we denote by $z=(u,v)$ a generic point on the stereographic plane. Straightforward analysis shows that 
\begin{subequations}\label{geodesic flow for 2-sphere}
\begin{equation}\label{z star 2sphere}
z^*(t;\eta)=2\,\tan(t/2)\,\dfrac{\eta}{\|\eta\|}\,,\\
\end{equation}
\begin{equation}\label{xi star 2sphere}
\xi^*(t;\eta)= \cos^2(t/2)\,\eta\,,
\end{equation}
\end{subequations}
provide a solution to the Hamiltonian system \eqref{Hamiltonian flow} for the Hamiltonian \eqref{hamiltonian laplacian} with initial conditions $z^*(0,\eta)=(0,0)$ and $\xi^*(0,\eta)=\eta=(\eta_1,\eta_2)$.

Our first goal is to compute the scalar part of the weight $w^2$ along the flow, i.e.\
\begin{equation*}
\left.       
\dfrac{\rho(y)}{\rho(z)}
\det\varphi_{z^\alpha\eta_\beta}
\right|_{z=z^*}\,,
\end{equation*}
for the Levi-Civita phase function $\varphi$ on the sphere associated with the metric $g$.

\begin{lemma}\label{lemma scalar part of w 2-sphere}
For the 2-sphere we have
\begin{equation}\label{scalar part of w 2-sphere}
\left.       
\dfrac{\rho(y)}{\rho(z)}
\det \,\varphi_{z^\alpha\eta_\beta}
\right|_{z=z^*}\,= \cos(t)-\ir\,\epsilon\,\sin(t).
\end{equation}
\end{lemma}

\begin{proof}
A key ingredient in the computation of \eqref{scalar part of w 2-sphere} is formula \eqref{phixeta at x=xstar} from Theorem~\ref{theorem phixeta at x=xstar}. As a first step, we need to compute the Christoffel symbols of $g$ along the geodesic flow. 

By means of \eqref{metric sphere stereographic coordinates} and \eqref{geodesic flow for 2-sphere}, one obtains
\begin{equation}\label{Christoffel 2-sphere along flow}
\begin{alignedat}{2}
\Gamma^u{}_{uu}(u^\ast,v^\ast)&= - \frac{\sin(t)}{2} \dfrac{\eta_u}{\|\eta \|},\quad&
\Gamma^{u}{}_{uv}(u^\ast,v^\ast)&= - \frac{\sin(t)}{2} \dfrac{\eta_v}{\|\eta \|}, \\
\Gamma^{u}{}_{vv}(u^\ast,v^\ast)&=\frac{\sin(t)}{2} \dfrac{\eta_u}{\|\eta \|},\quad&
\Gamma^v{}_{vv}(u^\ast,v^\ast)&=- \frac{\sin(t)}{2} \dfrac{\eta_v}{\|\eta \|},\\
\Gamma^{v}{}_{vu}(u^\ast,v^\ast)&=- \frac{\sin(t)}{2} \dfrac{\eta_u}{\|\eta \|},\quad&
\Gamma^{v}{}_{uu}(u^\ast,v^\ast)&= \frac{\sin(t)}{2} \dfrac{\eta_v}{\|\eta \|}\,.
\end{alignedat}
\end{equation}
Substituting \eqref{metric sphere stereographic coordinates}, \eqref{geodesic flow for 2-sphere} and \eqref{Christoffel 2-sphere along flow}
into \eqref{phixeta at x=xstar}, we get
\[
\begin{split}
&\left. \varphi_{z^\alpha\eta_\beta} \right|_{z=z^*}=\cos^2(t/2)\\
&\times
\begin{pmatrix}
1- [1-\cos(t)+\ir\,\epsilon\sin(t)] \dfrac{\eta_2^2}{\|\eta\|^2}
&
[1-\cos(t)+\ir\,\epsilon\sin(t)]\dfrac{\eta_1\eta_2}{\|\eta\|^2}\\
[1-\cos(t)+\ir\,\epsilon\sin(t)]\dfrac{\eta_1\eta_2}{\|\eta\|^2}
&
1- [1-\cos(t)+\ir\,\epsilon\sin(t)] \dfrac{\eta_1^2}{\|\eta\|^2}
\end{pmatrix}\,,
\end{split}
\]
from which it ensues that
\[
\left. \det \varphi_{z^\alpha\eta_\beta} \right|_{z=z^*}=\cos^4(t/2)  \left[\cos(t)-\ir\,\epsilon\sin(t) \right].
\]
Since $\rho(z^*(t;\eta))=\cos^4(t/2)$ and $\rho(y)=1$, this completes the proof.
\end{proof}

Note that \eqref{scalar part of w 2-sphere} is a scalar identity and, as such, independent of the choice of coordinates.

Let $\epsilon=0$, which corresponds to the adoption of a real-valued phase function.
Direct inspection of \eqref{scalar part of w 2-sphere} tells us that
$\left.\varphi_{z\eta}\right|_{z=z^*}$ becomes degenerate at $t=\frac\pi2+\pi k$,
$k\in\mathbb{Z}$ and, consequently, $w$ vanishes at these values of $t$.

If, on the other hand, $\epsilon>0$, then $w$ is non-zero for all values of $t$. This fact is the analytic counterpart of the circumvention of the obstruction caused by caustics.

The result of Lemma~\ref{lemma scalar part of w 2-sphere} can be used to compute the Maslov index. Let $\gamma$ be the lift to the Lagrangian submanifold $\Lambda_h$ of a great circle starting and ending at $y$ and set, for simplicity, $\epsilon=1$.  Then by \eqref{Maslov form}, \eqref{scalar part of w 2-sphere} we get
\[
\vartheta_\varphi=\dfrac{1}{\pi}\,\dr t
\]
and, in view of \eqref{Maslov index}, we conclude that
\[
\operatorname{ind}(\gamma)=\frac1\pi\int_0^{2\pi}\, \dr t=2\,.
\]

Let us now move on to the calculation of the $g$-subprincipal symbol of the wave propagator. For the 2-sphere the geodesic distance between two arbitrary points can be computed explicitly via a closed formula.
With the above notation, consider the auxiliary map
\begin{equation*}\label{distance 2-sphere step 1}
\tilde\sigma: \mathbb{R}^2\to \mathbb{R}^3,\qquad
(u,v) \mapsto \dfrac{1}{1+K^2}
\begin{pmatrix}
u\\
v\\
 K^2-1
\end{pmatrix},
\end{equation*}
which is nothing but the map \eqref{map sigma}
shifted by $(0,0,-1)$.
Then the geodesic distance between $(u,v)$ and $({u}',{v}')$ is given by
\begin{equation}\label{geodesic distance sphere}
\operatorname{dist}((u,v),({u}',{v}'))= \arccos \left[\, \tilde\sigma(u,v)\,\cdot\,\tilde\sigma({u}',{v}')\, \right],
\end{equation}
where the dot stands for the inner product in $\mathbb{E}^3$.

Formulae \eqref{geodesic distance sphere} and \eqref{Hamiltonian flow} yield an explicit representation for \eqref{phase function with distance}, which can be used to set up the algorithm described in Section~\ref{The subprincipal symbol of the propagator}.

For $\epsilon=1$ the functions appearing on the RHS of \eqref{ODE subprincipal symbol} read
\begin{subequations}
\begin{equation}\label{sigma 2 sphere}
{\mathfrak{S}}_{-2}\,b_2=\frac14 \left(-3+2\er^{2\,\ir\,t}+\er^{4\,\ir\,t}  \right),
\end{equation}
\begin{equation}\label{sigma 1 sphere}
{\mathfrak{S}}_{-1}\,b_1=\frac16 \left( 7-4\er^{2\,\ir\,t}-3 \er^{4\,\ir\,t} \right),
\end{equation}
\begin{equation}\label{sigma 0 sphere}
{\mathfrak{S}}_{0}\,b_0=\frac1{12} \left(-8 + \er^{2\,\ir\,t}\right).
\end{equation}
\end{subequations}
Substitution of \eqref{sigma 2 sphere}--\eqref{sigma 0 sphere}
into \eqref{ODE subprincipal symbol} yields a formula for the $g$-subprincipal symbol:
\begin{equation*}\label{subprincipal symbol 2-sphere epsilon equals 1}
\mathfrak{a}_{-1}(t;y,\eta;1)=
\dfrac{\ir\,t}{8\,\|\eta\|}+
\dfrac{2\er^{2\,\ir\,t}+3\er^{4\,\ir\,t}-5}{96\,\|\eta\|}\,.
\end{equation*}

For a general $\epsilon>0$ the corresponding formulae are more
complicated and the final expression for the $g$-subprincipal symbol reads
\begin{equation}\label{subprincipal symbol 2-sphere}
\mathfrak{a}_{-1}(t;y,\eta;\epsilon)=\dfrac{\ir\,t}{8\,\|\eta\|}+\dfrac{
\ir\sin(2t)-  4\epsilon\sin^2(t)        
+3 \,\ir\,\epsilon^2\sin(2t) +6\epsilon^3 \sin^2(t)}{48\,\|\eta\|\, (\cos (t)-\ir \epsilon  \sin (t))^2}\,.
\end{equation}

\begin{remark}
For $t\ne\pi/2+\pi k$, $k\in\mathbb{Z}$,
the $g$-subprincipal symbol admits the following expansion in powers of $\epsilon$:
\[
\begin{split}
\mathfrak{a}_{-1}(t;y,\eta;\epsilon)=
\dfrac{\ir\,t}{8\,\|\eta\|}&+\dfrac{\ir\,\tan(t)}{24\,\|\eta\|}
-\dfrac{\epsilon\,\tan^2(t)}{6\,\|\eta\|}\\
&+\ir\sum_{k=2}^\infty \dfrac{(\ir\,\epsilon)^k}{24\,\|\eta\|}\left( \tan(t) \right)^{k-1}\left( (3k+1)\tan^2(t)-3 \right).
\end{split}
\]
Note that for $\epsilon=0$ the above formula turns to
\begin{equation}
\label{subprincipal symbol 2-sphere epsilon 0}
\mathfrak{a}_{-1}(t;y,\eta;0)=
\frac \ir{24\,\|\eta\|}\left(3 t + \tan(t)\right),
\end{equation}
which is the $g$-subprincipal symbol of the propagator for
the real-valued Levi-Civita phase function.
Of course, formula
\eqref{subprincipal symbol 2-sphere epsilon 0}
can only be used for $t\in(-\pi/2,\pi/2)$:
topological obstructions prevent the use of the real-valued phase
function for large $t$.
It is easy to check that
\eqref{subprincipal symbol 2-sphere epsilon 0} agrees with
\eqref{subprincipal small times}, with $\mathcal{R}(y)=2$.
\end{remark}

\

Let us now run a test for our formula \eqref{subprincipal symbol 2-sphere}.
To this end, let us shift the Laplacian by a quarter,
\begin{equation}
\label{quarter test 1}
-\Delta\mapsto-\Delta+\frac14\,.
\end{equation}
Note that the eigenvalues of the operator $\,\sqrt{-\Delta+1/4}\,$ are half-integer, hence, the corresponding propagator
$\,\widetilde U(t):=\er^{-\ir t\sqrt{-\Delta+1/4}}\ $ is $2\pi$-antiperiodic,
\begin{equation}
\label{antiperiodic}
\widetilde U(t+2\pi)=-\widetilde U(t).
\end{equation}
Going back to Lemma~\ref{lemma decomposition of b}, we see that the shift of the Laplacian
\eqref{quarter test 1} does not affect $b_2$ and $b_1$, but shifts $b_0$ as
\begin{equation}
\label{quarter test 2}
b_0\mapsto b_0+1/4\,.
\end{equation}
Theorem~\ref{theorem ODE subprincipal symbol} and formula \eqref{quarter test 2} tell us that the
$g$-subprincipal symbol of the propagator transforms as
\begin{equation}
\label{quarter test 3}
\mathfrak{a}_{-1}(t;y,\eta;0)\mapsto \mathfrak{a}_{-1}(t;y,\eta;0)
-\frac{\ir\,t}{8\|\eta\|}\,.
\end{equation}
Applying the transformation  \eqref{quarter test 3} to formula \eqref{subprincipal symbol 2-sphere},
we see that the $g$-subprincipal symbol of the propagator becomes $2\pi$-periodic.
It remains only to reconcile the periodicity of the full symbol of the propagator with the antiperiodicity  \eqref{antiperiodic} of the propagator itself.
This is to do with the Maslov index: formulae
\eqref{weight w}
and
\eqref{scalar part of w 2-sphere}
tell us that the weight $w$ picks up a change of sign as we traverse the periodic geodesic, a great circle.

It is known that constructing the wave propagator associated with the shifted Laplacian \eqref{quarter test 1} is often easier and some formulae are available in the literature. For example, formulae for the wave kernel of shifted Laplacians on rank one symmetric spaces was computed in \cite{bunke}. See also \cite{cheeger}, \cite[Section~3]{smyshlyaev}. 

\

\subsection{The hyperbolic plane}
From a strictly rigorous point of view, our construction works for closed manifolds only. However, the compactness assumption is largely technical and can be relaxed, even though this generalisation is not absolutely straightforward. In the current paper we refrain from carrying out such an extension, but we discuss a non-compact example, formally applying our algorithm to the hyperbolic plane.

Adopting the hyperboloid model for the hyperbolic plane, we consider
the upper sheet of the hyperboloid
\[
\mathbb{H}^2:=\{ (x,y,z)\in\mathbb{R}^3\,|\,x^2+y^2-z^2=-1, \,	\,z>0   \}
\]
endowed with metric $\delta_\mathbb{H}=\dr x^2+\dr y^2-\dr z^2$.
Projecting $\mathbb{H}^2$ onto $\mathbb{R}^2$ with coordinates $(u,v)$, we obtain the induced metric
\[
g=\dfrac{1}{1+u^2+v^2}\left[(1+v^2)\,\dr u^2 - 2 uv \,\dr u\,\dr v+(1+u^2)\,\dr v^2 \right].
\]
The metric $g$ is Riemannian, with constant Gaussian curvature equal to $-1$.

Setting, without loss of generality, $y=0$ and denoting $z=(u,v)$, the cogeodesic flow is given by
\begin{subequations}\label{geodesic flow hyperbolic plane}
\begin{equation*}\label{zstar hyperbolic plane}
z^*(t;\eta)=\sinh (t)\, \dfrac{\eta}{\|\eta\|}\,,
\end{equation*}
\begin{equation*}\label{xistar hyperbolic plane}
\xi^*(t;\eta)=\dfrac{1}{\cosh(t)}\,\eta\,.
\end{equation*}
\end{subequations}
Unlike the sphere, the hyperbolic plane does not present caustics due to its negative curvature.
Hence, there are no obstructions to a construction global in time
with real-valued phase function.
In particular, the Levi-Civita phase function with $\epsilon=0$ can be used.

Arguing as for the 2-sphere, one gets for $\epsilon\ge0$
\[
\begin{split}
&\qquad\left.\varphi_{z^\alpha\eta_\beta} \right|_{z=x^\ast}=
\frac{1}{\|\eta\|^2}\\
&\times\begin{pmatrix}
\eta_1^2\operatorname{sech}(t)+\eta_2^2(\cosh(t)+\ir\,\epsilon\sinh(t)) 
&
-\eta_1\eta_2\tanh(t) (\sinh(t)+\ir\,\epsilon\cosh(t))
\\
-\eta_1\eta_2\tanh(t) (\sinh(t)+\ir\,\epsilon\cosh(t))
&
\eta_2^2\operatorname{sech}(t)+\eta_1^2(\cosh(t)+\ir\,\epsilon\sinh(t))
\end{pmatrix}
\end{split}
\]
and
\begin{equation*}
\label{scalar part of w squared hyperbolic plane}
\left.       
\dfrac{\rho(y)}{\rho(x)}
\det \,\varphi_{z^\alpha\eta_\beta}
\right|_{z=z^*}\,=\cosh(t)+\ir\,\epsilon\,\sinh(t).
\end{equation*}
Direct inspection immediately reveals that, as expected, $\left. \varphi_{z\eta}\right|_{z=z^*}$ is non-degenerate for all times, even with $\epsilon=0$.

Carrying out our algorithm for $\epsilon=0$, we establish
that the homogeneous components of the reduced amplitude read
\begin{subequations}
\begin{equation*}
\mathfrak{S}_{-2}\,a_2=-\frac23 (2 + \cosh(2 t)) \operatorname{sech}^2(t)\,,
\end{equation*}
\begin{equation*}
\mathfrak{S}_{-1}\,a_1=\frac23 (2 + \cosh(2 t)) \operatorname{sech}^2(t)\,,
\end{equation*}
\begin{equation*}
\mathfrak{S}_{0}\,a_0=\frac1{12} (3 + \operatorname{sech}^2(t))\,.
\end{equation*}
\end{subequations}
Substitution of the above expressions into
\eqref{ODE subprincipal symbol} yields a formula for the $g$-subprincipal symbol:
\begin{equation}\label{subprincipal symbol hyperbolic plane}
\mathfrak{a}_{-1}(t;y,\eta;0)=-\frac \ir{24\,\|\eta\|}\left(3 t + \tanh(t)\right).
\end{equation}

Note that formulae for the hyperbolic plane are very similar to those for the sphere, with trigonometric functions being replaced by their hyperbolic counterparts. This is consistent with the results in \cite{taylor}, see also \cite[Sec.~3.7.2]{Zel3}. Formula \eqref{subprincipal symbol hyperbolic plane} is, of course, in agreement with
\eqref{subprincipal small times}, with $\mathcal{R}(y)=-2$.

\

Our explicit examples gave us the opportunity to illustrate, once again, the importance of formula \eqref{phixeta at x=xstar}: it allows one to extract topological information by means of a simple direct computation. 

\section{Circumventing the obstructions: a geometric picture}
\label{Circumventing topological obstructions: a geometric picture}

As discussed in the previous sections, the weight $w$
defined by formula \eqref{weight w}
is a crucial object in our mathematical construction in that it carries important information about $\Lambda_h$. It is possible, for instance, to compute the Maslov index purely in terms of $w$. The fact that, in general, a construction global in time is impossible using real-valued phase functions can be traced back to the degeneracy of $w$. In this section we will provide a geometric description of $\varphi_{x\eta}\,$, the key ingredient of $w$, along the flow.

Let us fix a point $y\in M$ and consider the one-parameter family of $d$-dimensional smooth submanifolds of the cotangent bundle defined by
\begin{equation*}
\mathcal{T}_{y}(t):=\{ (x^*(t;y,\eta),\xi^*(t;y,\eta))\in T^*M\, |\, \eta\in T'_y M  \}.
\end{equation*}
For every value of $t\,$, $\mathcal{T}_y(t)$ consists of all points of the cotangent bundle corresponding to the cogeodesic flow at time $t$ for the initial position $y$ and all possible momenta.
The smoothness of $\mathcal{T}_y(t)$ follows, for example, from the preservation of
the symplectic volume.

The manifolds $\mathcal{T}_y(t)$ are Lagrangian. In fact, $\mathcal{T}_y(0)=T'_yM=T^*_yM\setminus\{0\}$
is the punctured cotangent fibre at $y$, which is clearly Lagrangian, and the cogeodesic flow preserves the symplectic form.

In the following we will construct a family of metrics associated with the above submanifolds. In the rest of this section we will drop the arguments $t$ and $y$ in $x^*$ and $\xi^*$ whenever these arguments are fixed, writing simply $x^*(\eta)$ and $\xi^*(\eta)$.

In an arbitrary coordinate system a small increment $\delta \eta$ in momentum produces an increment in $x^*(\eta)$ given by
\begin{equation*}
\begin{split}
[x^*(\eta+\delta\eta)-x^*(\eta)]^\alpha=[x^*(\eta)]^\alpha_{\eta_\mu}\,\delta\eta_\mu +O(\|\delta\eta\|^2).
\end{split}
\end{equation*}
This allows us to define a bilinear form
\begin{equation}\label{position form}
\begin{split}
Q^{\mu\nu}(\eta;t,y):= g_{\alpha\beta}(x^*(\eta))\,q^{\alpha\mu}(\eta;t,y)\,q^{\beta\nu}(\eta;t,y)\,,
\end{split}
\end{equation}
where 
\begin{equation}
\label{q}
q^{\alpha\mu}(\eta;t,y):=[x^*(\eta)]^\alpha_{\eta_\mu}\,.
\end{equation}
We call $Q$ \emph{the position form}.

An analogous construction is possible for momentum $\xi^*(\eta)$, although extra care is needed due to the fact that $\xi^*(\eta)$ and $\xi^*(\eta+\delta\eta)$ live in different fibres of the bundle. Under the assumption that $\delta\eta$ is sufficiently small, let us parallel transport $\xi^*(\eta+\delta\eta)$ along the (unique) geodesic going from $x^*(\eta+\delta\eta)$ to $x^*(\eta)$, denoted by $\gamma:[0,1]\to M$. The parallel transport equation reads
\begin{equation}\label{parallel transport momentum eta+deta}
\dot \gamma^\alpha(s)\,\nabla_\alpha\,\zeta(\gamma(s))_\beta=\dot \gamma^\alpha(s)\,[\partial_\alpha \zeta_\beta(\gamma(s))- \Gamma^{\rho}{}_{\alpha\beta}(\gamma(s))\,\zeta_\rho(\gamma(s))]= 0\,,
\end{equation}
where $\zeta$ denotes the image under parallel transport of $\xi^*(\eta+\delta\eta)$ along $\gamma$.
It is not hard to check that the solution to \eqref{parallel transport momentum eta+deta} is given by
\[
\zeta_\alpha(\gamma(s))=\xi^*_{\alpha}(\eta+\delta \eta)+\Gamma^\rho{}_{\alpha\beta}(\gamma(s))\,\xi^*_\rho(\eta)\,s\,(\delta x^*)^\beta +O(\| \delta x^* \|^2)\,,
\]
where $\delta x^*=x^*(\eta)- x^*(\eta+\delta \eta)$. Hence, we get
\[
\begin{split}
\zeta_\alpha(\gamma(1))-\xi^*_\alpha(\eta)&= \left[(\xi^*_\alpha(\eta))_{\eta_\mu}- \Gamma^\rho{}_{\alpha\beta}(x^*(\eta))\,\xi^*_\rho(\eta)\,(x^*(\eta))^\beta_{\eta_\mu}  \right]\,(\delta\eta)_\mu\\
&+O(\|\delta\eta\|^2)\,.
\end{split}
\]
Put
\begin{equation}\label{p}
p^{\alpha\mu}(\eta;t,y):=g^{\alpha\gamma}(x^*(\eta))
\left[(\xi^*_\gamma(\eta))_{\eta_\mu}- \Gamma^\rho{}_{\gamma\beta}(x^*(\eta))\,\xi^*_\rho(\eta)\,(x^*(\eta))^\beta_{\eta_\mu} \right]
\end{equation}
and define the bilinear form
\begin{equation}\label{momentum form}
P^{\mu\nu}(\eta;t,y):=g_{\alpha\beta}(x^*(\eta))\, p^{\alpha\mu}(\eta;t,y)\,p^{\beta\nu}(\eta;t,y)\,.
\end{equation}
We call $P$ \emph{the momentum form}.

It is convenient, at this point, to redefine the position and momentum forms by lowering their indices using the metric $g$ at the point $y$. Hence, further on we have $Q=Q_{\mu\nu}$ and $P=P_{\mu\nu}$. Clearly, by construction, we have
\begin{equation*}
Q,P \in C^\infty(\mathcal{T}_y(t);\otimes_s^2 T^*\mathcal{T}_y(t)).
\end{equation*}
Our $Q$ and $P$ are natural candidates for metrics on $\mathcal{T}_y(t)$. This turns out not to be the case: $P$ and $Q$ are pseudometrics but not necessarily metrics. However, their sum is a metric.

\begin{theorem}
Let $a$ and $b$ be positive parameters.
Then the  linear combination of the position and momentum forms
\begin{equation}
\label{formula with a and b}
ah^2Q+bP\in C^\infty(\mathcal{T}_y(t);\otimes^2_s T^*\mathcal{T}_y(t))
\end{equation}
is a metric.
\end{theorem}

The $h$ in the above formula stands for $h(y,\eta)$. This factor has been introduced so that both
terms have the same degree of homogeneity (zero) in $\eta$.

\begin{proof}
Our $Q$ and $P$ are symmetric and can be written as
$Q=q^Tg\,q$, $P=p^Tg\,p$,
which implies that they are non-negative.
To prove that their linear combination
$ah^2Q+bP=ah^2q^Tg\,q+b\,p^Tg\,p\,$
is a metric we only need show that it is non-degenerate.
Choosing normal geodesic coordinates $x$ centred at $x^*(t;y,\eta)$, 
it is easy to see that $v\in T_{(x^*(\eta),\xi^*(\eta))}\mathcal{T}_y(t)$ is in the null space of $ah^2Q+bP$ 
if and only if $v^\flat$ satisfies
\begin{equation}\label{nondegeneracy P+Q vs flow}
[x^*(\eta)]^\alpha_{\eta_\mu}\,v_\mu=0 
\qquad \text{and}\qquad 
[\xi^*_\alpha(\eta)]_{\eta_\mu}\,v_\mu=0\,.
\end{equation}
Since the Hamiltonian flow is non-degenerate, i.e.~it preserves the tautological $1$-form, the two conditions \eqref{nondegeneracy P+Q vs flow} cannot be simultaneously fulfilled unless $v=0$. Therefore, $ah^2Q+bP$ is non-degenerate.
\end{proof}

The metric $ah^2Q+bP$ is closely related to $\varphi_{x\eta}$ along the flow: condition (iii) in Definition \ref{phase function of class Fh} translates, in geometric terms, into the statement that the intersection of null spaces of $Q$ and $P$ is the zero subspace.
The weight $w$ becoming degenerate in the case of a real-valued phase function corresponds, in this geometric picture, to $Q$ and $P$ separately not being metrics. We will show this below for the case of the 2-sphere, as an explicit example.

Before moving to that, let us make the aforementioned relation between $Q$, $P$ on the one hand and $\varphi_{x\eta}$ on the other mathematically precise.

\begin{theorem}\label{theorem geometric characterisation phixeta}
We have
\begin{equation}\label{relation between Phixeta and family of metrics}
\left.\varphi_{x^\alpha\eta_\mu}\right|_{x=x^*}=g_{\alpha\beta}(x^*)\left[ p^{\beta\mu}- \ir\,\epsilon\,h\,q^{\beta\mu} \right].
\end{equation}
\end{theorem}

\begin{proof}[Proof of Theorem \ref{theorem geometric characterisation phixeta}]
The identity \eqref{relation between Phixeta and family of metrics} is established by comparing \eqref{phixeta at x=xstar} with \eqref{q} and \eqref{p}.
\end{proof}

\begin{example}[Position and momentum forms for $\mathbb{S}^2$]
With the notation of Section~\ref{Explicit examples}, the quantities $q$ and $p$
defined by formulae
\eqref{q} and \eqref{p}
read
\begin{equation*}
\begin{aligned}
q^{\alpha\mu}&=\dfrac{2\, \tan(t/2)}{\|\eta\|^3}\begin{pmatrix}
\eta_2^2 & -\eta_1 \eta_2\\
-\eta_1\eta_2 & \eta_1^2
\end{pmatrix},
\\
p^{\alpha\mu}&=\dfrac{1}{\cos^2(t/2)\,\|\eta\|^2}\begin{pmatrix}
\eta_1^2+\eta_2^2\cos(t) &\eta_1\eta_2(1-\cos(t))\\
\eta_1\eta_2(1-\cos(t)) &\eta_2^2+\eta_1^2\cos(t)
\end{pmatrix}.
\end{aligned}
\end{equation*}
Consequently, the position and momentum forms are given by
\[
\begin{split}
Q_{\mu\nu}&=\dfrac{\sin^2(t)}{\| \eta \| ^4}
\begin{pmatrix}
 \eta_2^2&- \eta_1\,\eta_2\\
-\eta_1\,\eta_2& \eta_1^2
\end{pmatrix},
\\
P_{\mu\nu}&=\dfrac{1}{\| \eta \|^2}
\begin{pmatrix}
\eta_1^2+\eta_2^2\,\cos^2(t)& \eta_1\,\eta_2\,\sin^2(t)\\
\eta_1\,\eta_2\,\sin^2(t)& \eta_2^2+\eta_1^2\,\cos^2(t)
\end{pmatrix}.
\end{split}
\]

We have $\det Q=0$ and $\det P=\cos^2(t)$.
This implies that $P$, which is associated with the real part of \eqref{phixeta at x=xstar}
via \eqref{relation between Phixeta and family of metrics}
and \eqref{momentum form},
becomes degenerate for $t=\pi/2$. However, for the full metric $h^2Q+P$ we have 
in chosen local coordinates $h^2Q_{\mu\nu}+P_{\mu\nu}=\delta_{\mu\nu}\,$, so that
that the full metric $h^2Q+P$ is non-degenerate for all $t\in\mathbb{R}$.
This example
is remarkable in that the metric
\eqref{formula with a and b}
with $a=b=1$
does not depend on $t$.
\end{example}

\section*{Acknowledgements}
\addcontentsline{toc}{section}{Acknowledgements}

We are grateful to Jeff Galkowski and Valery Smyshlyaev for stimulating discussions and to Steve Zelditch for valuable comments. We would also like to thank the anonymous referees for a number of useful suggestions. DV was supported by EPSRC grant EP/M000079/1.

\begin{appendices}

\section{The amplitude-to-symbol operator}
\label{The amplitude-to-symbol operator}

In this appendix we will provide mathematical proofs and rigorous justification to the amplitude reduction algorithm described in Section~\ref{The global invariant symbol of the propagator}, developing ideas outlined in \cite{SaVa}.

With the notation established throughout the paper, let
$a\in S^m_{\mathrm{ph}}(\mathbb{R}\times M \times T'M)$
be a poly\-homogeneous function of order $m$,
\[
a\sim \sum_{k=0}^\infty a_{m-k}\,.
\]
Consider the oscillatory integral
\begin{equation}\label{oscillatory_w_amplitude}
\mathcal{I}_\varphi(a)=\int_{T^*_yM} \er^{\ir \varphi(t,x;y,\eta)} \, a(t,x;y,\eta)\, w(t,x;y,\eta) \,\dbar\eta\,,
\end{equation}
where $\varphi$ is any phase function of class $\mathcal{L}_h$.
For the sake of clarity, we drop here the dependence of functions on extra parameters (e.g.~$\epsilon$).

It is a well known fact that, modulo an infinitely smooth contribution,
\begin{equation}\label{oscillatory_w_symbol}
\mathcal{I}_\varphi(a)\overset{\text{mod}\,C^\infty}{=}\int_{T^*_yM} \er^{\ir \varphi(t,x;y,\eta)} \, \mathfrak{a}(t;y,\eta)\, w(t,x;y,\eta) \,\dbar\eta,
\end{equation}
for some $\mathfrak{a}=\mathfrak{a}(t;y,\eta)$. We call the $a$ in \eqref{oscillatory_w_amplitude}  \emph{amplitude} and
the $\mathfrak{a}$ in \eqref{oscillatory_w_symbol} \emph{symbol}. 

In this framework, one can construct an amplitude-to-symbol operator
\begin{equation*}
\mathfrak{S}: a\mapsto \mathfrak{a}.
\end{equation*}
The aim of this appendix is to write down the operator $\mathfrak{S}$ explicitly.

\begin{theorem}\label{main_proposition}
The amplitude-to-symbol operator $\mathfrak{S}$ reads
\begin{equation}\label{mathfrak_s}
\mathfrak{S}\sim \sum_{k=0}^\infty \mathfrak{S}_{-k},
\end{equation}
where
\begin{gather}
\mathfrak{S}_0=\left.\left( \,\cdot\, \right)\right|_{x=x^*}, \label{mathfrak_s0}\\
\mathfrak{S}_{-k}=\mathfrak{S}_0 \left[\ir\, w^{-1} \frac{\partial}{\partial \eta_\beta}\, w \left( 1+ \sum_{1\leq |\boldsymbol{\alpha}|\leq 2k-1} \dfrac{(-\varphi_\eta)^{\boldsymbol{\alpha}}}{\boldsymbol{\alpha}!\,(|\boldsymbol{\alpha}|+1)}\,L_{\boldsymbol{\alpha}} \right)  L_\beta  \right]^k\label{mathfrak_sk}
\end{gather}
with $L_\alpha:=[(\varphi_{x\eta})^{-1}]_\alpha{}^\beta \dfrac{\partial}{\partial x^\beta}$.
\end{theorem}

We begin with two general comments regarding our phase function, which follow from the properties in Definition~\ref{phase function of class Fh}. Firstly, as already observed, $\varphi_\eta(t,x^*;y,\eta)=0$. Secondly, one can always assume that $\det(\varphi_{x^\alpha\eta_\beta})\ne 0 $ on $\operatorname{supp} a$. If this is not the case, it is enough to multiply $a$ by a smooth cut-off $\chi$ supported in a neighbourhood of
\[
\mathfrak{C}=\{(t,x;y,\eta)\,\,|\,\,x=x^*(t;y,\eta) \} \subset \mathbb{R}\times M \times T'M
\]
small enough. The oscillatory integrals $\mathcal{I}_\varphi(a)$ and $\mathcal{I}_\varphi(\chi \,a)$ differ by infinitely smooth contributions.

The idea of the proof, at times quite technical, goes as follows. Expand the amplitude $a$ in power series in $x$ about $x=x^*$. With the notation $a^*=a|_{x=x^*}$, we have
\begin{equation}\label{a_expansion}
a=a^*+ (x-x^*)^\alpha\, b_\alpha
\end{equation}
for some covector $b=b(t,x;y,\eta)$. Plugging \eqref{a_expansion} into \eqref{oscillatory_w_amplitude}, we obtain
\begin{equation}\label{amplitude reduction euristics}
\begin{split}
\mathcal{I}_\varphi(a)&=\int_{T'_yM}  \er^{\ir \varphi} \, a^*\, w \,\dbar\eta + \int_{T'_yM}  \er^{\ir \varphi} \,(x-x^*)^\alpha\, b_\alpha\, w \,\dbar\eta\\
&=\int_{T'_yM}  \er^{\ir \varphi} \, a^*\, w \,\dbar\eta + \int_{T'_yM} 
\er^{\ir \varphi} \,\varphi_{\eta_\alpha}\, \tilde b_\alpha\, w \,\dbar\eta\\
&=\int_{T'_yM}  \er^{\ir \varphi} \, a^*\, w \,\dbar\eta +\int_{T'_yM}  \frac1\ir
\left(\dfrac{\partial}{\partial \eta_\alpha}  \er^{\ir \varphi}\right)
\tilde b_\alpha\, w \,\dbar\eta\\
&=\int_{T'_yM}  \er^{\ir \varphi} \, a^*\, w \,\dbar\eta +\int_{T'_yM}  \er^{\ir \varphi} \, \ir\,w^{-1} 
 \left( \dfrac{\partial}{\partial \eta_\alpha}  \tilde b_\alpha\,w  \right)
w \,\dbar\eta\,,
\end{split}
\end{equation}
where the covector $\tilde{b}$ can be written down explicitly in terms of $b$ and $\varphi$.
It is easy to see that
\[  
 \ir\,w^{-1}  \left(  \dfrac{\partial}{\partial \eta_\alpha} \tilde b_\alpha\,w \right) \in
 S^{m-1}_{\mathrm{ph}}(\mathbb{R}\times M \times T'M).
\]
The first integral on the RHS of \eqref{amplitude reduction euristics} has amplitude independent of $x$, whereas the second one has amplitude whose order is decreased by one. Repeating the above argument, we can recursively reduce the order and eventually obtain an oscillatory integral with $x$-independent amplitude
\[
\mathfrak{a} \sim \sum_{k=0}^\infty \mathfrak{a}_{m-k}, \quad \mathfrak{a}_{m-k} \in
 S^{m-k}_{\mathrm{ph}}(\mathbb{R}\times T'M),
\]
plus an oscillatory integral with amplitude in $ S^{-\infty}(\mathbb{R}\times M \times T'M)$.

Note that the $b$ and $\tilde b$ in the above argument are both covectors but in a different sense:
$b_\alpha$ behaves as a covector under changes of local coordinates $x$, whereas
$\tilde b_\alpha$ behaves as a covector under changes of local coordinates $y$.

The actual proof relies on a more sophisticated argument, which allows one to explicitly and constructively compute $\mathfrak{a}$. The whole idea, rooted in a version of the Malgrange preparation theorem, is to factor out $\varphi_{\eta_\alpha}$ rather than simply $(x-x^*)^\alpha$ in equation~\eqref{a_expansion}. Such factorisation will be eventually achieved in formula \eqref{referee 3}. A crucial point worth stressing is that the whole construction is global and covariant.

Before addressing the proof
of Theorem~\ref{main_proposition}
we need to state and prove a preparatory lemma.

\begin{lemma}\label{lem:commutativity_Ls}
The operators
\begin{equation}\label{operator_L}
L_\alpha= (\varphi_{x\eta}^{-1})_\alpha{}^\gamma \,\dfrac{\partial}{\partial x^\gamma}
\end{equation}
commute. Namely, for all $\,\alpha, \beta=1,\ldots,d\,$ we have
\begin{equation}\label{L_commute}
[L_\alpha,L_\beta]=0.
\end{equation}
\end{lemma}

\begin{proof}
We have
\begin{equation*}
\begin{split}
L_\alpha L_\beta - L_\beta L_\alpha
&=(\varphi_{x\eta}^{-1})_\alpha{}^\mu \,\dfrac{\partial}{\partial x^\mu} (\varphi_{x\eta}^{-1})_\beta{}^\nu \,\dfrac{\partial}{\partial x^\nu}
-(\varphi_{x\eta}^{-1})_\beta{}^\nu \,\dfrac{\partial}{\partial x^\nu}(\varphi_{x\eta}^{-1})_\alpha{}^\mu \,\dfrac{\partial}{\partial x^\mu}
\\
&=
\left((\varphi_{x\eta}^{-1})_\alpha{}^\mu [(\varphi_{x\eta}^{-1})_\beta{}^\nu]_{x^\mu} \right)\dfrac{\partial}{\partial x^\nu}
- 
\left( (\varphi_{x\eta}^{-1})_\beta{}^\nu [(\varphi_{x\eta}^{-1})_\alpha{}^\mu]_{x^\nu} \right)\dfrac{\partial}{\partial x^\mu}
\\
&=
\left(
(\varphi_{x\eta}^{-1})_\alpha{}^\nu [(\varphi_{x\eta}^{-1})_\beta{}^\mu]_{x^\nu}
-
(\varphi_{x\eta}^{-1})_\beta{}^\nu [(\varphi_{x\eta}^{-1})_\alpha{}^\mu]_{x^\nu}
\right)
\dfrac{\partial}{\partial x^\mu}\,.
\end{split}
\end{equation*}
Contracting with $(\varphi_{x\eta})_\gamma{}^\alpha\, (\varphi_{x\eta})_\rho{}^\beta$, we get
\begin{equation}\label{L_commute_proof}
\begin{split}
(\varphi_{x\eta})_\gamma{}^\alpha\, (\varphi_{x\eta})_\rho{}^\beta\,[L_\alpha,L_\beta] 
&= \left(
(\varphi_{x\eta})_\rho{}^\beta[(\varphi_{x\eta}^{-1})_\beta{}^\mu]_{x^\gamma}
-
 (\varphi_{x\eta})_\gamma{}^\alpha \,[(\varphi_{x\eta}^{-1})_\alpha{}^\mu]_{x^\rho}
\right)
\dfrac{\partial}{\partial x^\mu}
\\
&= \left(
- 
[(\varphi_{x\eta})_\rho{}^\beta]_{x^\gamma}(\varphi_{x\eta}^{-1})_\beta{}^\mu
+ 
[(\varphi_{x\eta})_\gamma{}^\alpha]_{x^\rho} \,(\varphi_{x\eta}^{-1})_\alpha{}^\mu
\right)
\dfrac{\partial}{\partial x^\mu}
\\
&= \left(
-
\varphi_{x^\rho x^\gamma \eta_\beta}\, (\varphi_{x\eta}^{-1})_\beta{}^\mu
+
\varphi_{x^\gamma x^\rho \eta_\alpha} \,(\varphi_{x\eta}^{-1})_\alpha{}^\mu
\right)
\dfrac{\partial}{\partial x^\mu}
\\
&= \left(
-\varphi_{x^\rho x^\gamma \eta_\alpha}\, (\varphi_{x\eta}^{-1})_\alpha{}^\mu
+
\varphi_{x^\gamma x^\rho \eta_\alpha} \,(\varphi_{x\eta}^{-1})_\alpha{}^\mu
\right)
\dfrac{\partial}{\partial x^\mu}
\\
&=0.
\end{split}
\end{equation}
Since $\varphi_{x\eta}$ is non-degenerate, \eqref{L_commute_proof} is equivalent to \eqref{L_commute}.
\end{proof}

We are now in a position to prove Theorem \ref{main_proposition}.

\begin{proof}[Proof of Theorem \ref{main_proposition}]
The first step is to show that it is possible to write, modulo $O(\|x-x^*\|^\infty)$,  the amplitude $a$ as
\begin{equation}
\label{first step}
a(t,x;y,\eta)=a(t,x^*(t;y,\eta);y,\eta) + \varphi_{\eta_\alpha}(t,x;y,\eta)\, \tilde b_\alpha(t,x; y,\eta)
\end{equation}
for some $\tilde b$.

In order to write down explicitly the $\tilde b$ appearing in formula \eqref{first step},
let us introduce the operators
\begin{equation}\label{operator_F0}
F_0:=1,
\end{equation}
\begin{equation}\label{operator_Fk}
F_k:=\sum_{|\boldsymbol{\alpha}|=k} \dfrac{(\varphi_\eta)^{\boldsymbol{\alpha}}}{\boldsymbol{\alpha}!} L_{\boldsymbol{\alpha}}\,,
\end{equation}
where ${\boldsymbol{\alpha}}=(\alpha_1,\alpha_2,\ldots,\alpha_d)\in \mathbb{N}_0^d$ is a multi-index, 
${\boldsymbol{\alpha}}!=\alpha_1!\,\alpha_2! \cdots \alpha_d!\,$, 
 $(\varphi_\eta)^{\boldsymbol{\alpha}}=(\varphi_{\eta_1})^{\alpha_1}(\varphi_{\eta_2})^{\alpha_2}\cdots(\varphi_{\eta_d})^{\alpha_d}$, 
$L_{\boldsymbol{\alpha}}=(L_1)^{\alpha_1}(L_2)^{\alpha_2}\cdots(L_d)^{\alpha_d}$.
In view of Lem\-ma~\ref{lem:commutativity_Ls}, $F_k$ is well defined and the order of the $L_\alpha$'s is irrelevant.
Note also that the coefficients $\frac1{\boldsymbol{\alpha}!}$ appearing in \eqref{operator_Fk}
are the ones from the algebraic multinomial expansion
\begin{equation}\label{algebraic multinomial expansion}
(z_1+\ldots+z_d)^k=
k!
\sum_{|\boldsymbol{\alpha}|=k}
\dfrac1{\boldsymbol{\alpha}!}\,z^{\boldsymbol{\alpha}},
\end{equation}
a generalisation of the binomial expansion.

Formulae \eqref{operator_Fk} and \eqref{algebraic multinomial expansion} imply
\begin{equation}\label{relation_Fk+1_Fk temp 1}
(k+1)F_{k+1}=\sum_{\gamma=1}^d \varphi_{\eta_\gamma}F _k \, L_\gamma\,.
\end{equation}
Furthermore, we have
\begin{equation}
\label{relation_Fk+1_Fk temp 2}
\begin{split}
F_1\, F_k- k\, F_k &= \left(\sum_{\gamma=1}^d \varphi_{\eta_\gamma} \,L_\gamma\right) F_k - k\,F_k\\
&= \sum_{\gamma,\mu=1}^d \varphi_{\eta_\gamma} (\varphi_{x\eta}^{-1})_\gamma{}^\mu \sum_{|{\boldsymbol{\alpha}}|=k} \dfrac{\left[(\varphi_\eta)^{\boldsymbol{\alpha}} \right]_{x^\mu}}{{\boldsymbol{\alpha}}!} L_{\boldsymbol{\alpha}}\\
&\qquad +\sum_{\gamma=1}^d \varphi_{\eta_\gamma}  \sum_{|{\boldsymbol{\alpha}}|=k} \dfrac{(\varphi_\eta)^{\boldsymbol{\alpha}}}{{\boldsymbol{\alpha}}!} L_\gamma L_{\boldsymbol{\alpha}} - k\, F_k\\
&= k\, F_k + \sum_{\gamma=1}^d \varphi_{\eta_\gamma}  \sum_{|{\boldsymbol{\alpha}}|=k} \dfrac{(\varphi_\eta)^{\boldsymbol{\alpha}}}{{\boldsymbol{\alpha}}!}L_{\boldsymbol{\alpha}} L_\gamma - k\, F_k\\
&=\sum_{\gamma=1}^d \varphi_{\eta_\gamma}F _k \, L_\gamma\,.
\end{split}
\end{equation}
Combining formulae
\eqref{relation_Fk+1_Fk temp 1}
and
\eqref{relation_Fk+1_Fk temp 2},
we arrive at a recurrent formula for our operators
$F_k\,$:
\begin{equation}\label{relation_Fk+1_Fk}
(k+1)F_{k+1}=F_1\, F_k- k\, F_k\,.
\end{equation}

It turns out that the functions $(\varphi_\eta)^{\boldsymbol{\alpha}}$
with $|{\boldsymbol{\alpha}}|\ge k$
are eigenfunctions of the
operators $F_k$. Namely, we have
\begin{equation}\label{homogeneous_property_s_r}
F_k (\varphi_\eta)^{\boldsymbol{\alpha}}
= 
\begin{cases}
\quad\,\,\,\,\,\, 0, \qquad\qquad |{\boldsymbol{\alpha}}|< k,\\
{|{\boldsymbol{\alpha}}|\choose k}
(\varphi_\eta)^{\boldsymbol{\alpha}},\qquad |{\boldsymbol{\alpha}}|\ge k.
\end{cases}
\end{equation}
Formula \eqref{homogeneous_property_s_r} can be proved by induction. It is clearly true for $k=0$. Let us assume it is true for $k=n$. Let us prove it for $k=n+1$.
If $|{\boldsymbol{\alpha}}|<n$, then the required result immediately follows from
formula \eqref{relation_Fk+1_Fk} and the inductive assumption.
If $|{\boldsymbol{\alpha}}|\ge n$, then
formula \eqref{relation_Fk+1_Fk} and the inductive assumption give us
\begin{equation*}
\begin{split}
F_{n+1}(\varphi_\eta)^{\boldsymbol{\alpha}}&= \frac{1}{n+1}
{|{\boldsymbol{\alpha}}|\choose n}
\left[
 F_1(\varphi_\eta)^{\boldsymbol{\alpha}}
-
n(\varphi_\eta)^{\boldsymbol{\alpha}}
\right]\\
&= \frac{1}{n+1}
{|{\boldsymbol{\alpha}}|\choose n}
\left[
|{\boldsymbol{\alpha}}|(\varphi_\eta)^{\boldsymbol{\alpha}}
-
n(\varphi_\eta)^{\boldsymbol{\alpha}}
\right]
\\
&= \frac{|{\boldsymbol{\alpha}}|-n}{n+1}
{|{\boldsymbol{\alpha}}|\choose n}
(\varphi_\eta)^{\boldsymbol{\alpha}}
=
\begin{cases}
\qquad 0,\qquad\qquad\,\, |{\boldsymbol{\alpha}}|=n,\\
{|{\boldsymbol{\alpha}}|\choose n+1}
(\varphi_\eta)^{\boldsymbol{\alpha}},\qquad |{\boldsymbol{\alpha}}|> n,
\end{cases}
\end{split}
\end{equation*}
as required.

Formula \eqref{homogeneous_property_s_r} is, effectively, a generalised version of
Euler's formula for homogeneous functions.

Given a multi-index ${\boldsymbol{\alpha}}\ne0$, we have the elementary identity
\[
0=(1-1)^{|{\boldsymbol{\alpha}}|}
=\sum_{k=0}^{|{\boldsymbol{\alpha}}|}
(-1)^k
{|{\boldsymbol{\alpha}}|\choose k}
=1+
\sum_{k=1}^{|{\boldsymbol{\alpha}}|}
(-1)^k
{|{\boldsymbol{\alpha}}|\choose k}.
\]
The above identity and formula \eqref{homogeneous_property_s_r}  imply
\begin{equation}
\label{appendix a dima 1}
(\varphi_\eta)^{\boldsymbol{\alpha}}
=
-
\left(
\sum_{k=1}^{|{\boldsymbol{\alpha}}|}
(-1)^kF_k
\right)
(\varphi_\eta)^{\boldsymbol{\alpha}}
=
-
\left(
\sum_{k=1}^\infty
(-1)^kF_k
\right)
(\varphi_\eta)^{\boldsymbol{\alpha}},
\qquad\forall{\boldsymbol{\alpha}}\ne0.
\end{equation}

Consider now a function $a(t,x;y,\eta)$.
It can be expanded into an asymptotic series in powers of $x-x^*$.
Observe that $\varphi_\eta$ can also be expanded into an asymptotic series in powers of $x-x^*$
and, furthermore, in view of Definition~\ref{phase function of class Fh} this series can be inverted,
giving an asymptotic expansion of  $x-x^*$ in powers of $\varphi_\eta\,$.
Consequently, the function $a(t,x;y,\eta)$
can be expanded into an asymptotic series in powers of $\varphi_\eta\,$.
The coefficients of the latter expansion are determined using the fact that
\[
\left.
\bigl[
L_{\boldsymbol{\alpha}}
(\varphi_\eta)^{\boldsymbol{\beta}}
\,\bigr]
\right|_{x=x^*}=
\begin{cases}
{\boldsymbol{\alpha}}!,\qquad{\boldsymbol{\alpha}}={\boldsymbol{\beta}},
\\
\,\,0,\qquad\,\,{\boldsymbol{\alpha}}\ne{\boldsymbol{\beta}}.
\end{cases}
\]
This gives us
\begin{equation}
\label{expansion for a in powers of varphi eta}
a\simeq\sum_{|{\boldsymbol{\alpha}}|\ge0}
 \dfrac{(\varphi_\eta)^{\boldsymbol{\alpha}}}{{\boldsymbol{\alpha}}!} \left.[L_{\boldsymbol{\alpha}}a]\right|_{x=x^*}.
\end{equation}
The symbol $\simeq$ in \eqref{expansion for a in powers of varphi eta} indicates that we are dealing with
an asymptotic expansion. Namely, it means that for any $r\in\mathbb{N}_0$ we have
\[
a-\sum_{0\le|{\boldsymbol{\alpha}}|\le r}
 \dfrac{(\varphi_\eta)^{\boldsymbol{\alpha}}}{{\boldsymbol{\alpha}}!} \left.[L_{\boldsymbol{\alpha}}a]\right|_{x=x^*}
=O\bigl(\|x-x^*\|^{r+1}\bigr).
\]

Formula \eqref{appendix a dima 1} allows us to rewrite the asymptotic expansion
\eqref{expansion for a in powers of varphi eta} as
\begin{equation}
\label{expansion for a in powers of varphi eta rewritten}
a\simeq a|_{x=x^*}-\sum_{k=1}^\infty (-1)^kF_k\,a.
\end{equation}
The advantage of
\eqref{expansion for a in powers of varphi eta rewritten}
over
\eqref{expansion for a in powers of varphi eta}
is that the restriction operator $\left.(\,\cdot\,)\right|_{x=x^*}$
appears only in one place, in the first term on the RHS of
\eqref{expansion for a in powers of varphi eta rewritten}.
Formula
\eqref{expansion for a in powers of varphi eta rewritten}
is a generalisation of the formula
\begin{equation}
\label{expansion for a in powers of varphi eta rewritten 1D}
a(x)\simeq
a(0)
+xa'(x)
-
\frac{x^2}2a''(x)
+
\frac{x^3}6a'''(x)
+\ldots
\end{equation}
from the analysis of functions of one variable.
Namely, formula
\eqref{expansion for a in powers of varphi eta rewritten}
turns into
\eqref{expansion for a in powers of varphi eta rewritten 1D}
if we set $d=1$ and choose a phase function $\varphi$ linear in $x$.

At this point it is worth discussing what happens under changes of local coordinates $x$.
Examination of formula \eqref{operator_L} shows that the operators $L_\alpha$
map scalar functions to scalar functions, i.e.~the map $a\mapsto L_\alpha a$
is invariant under changes of local coordinates $x$; note that the index $\alpha$
does not play a role in this argument as it lives at a different point, $y$, and in a different coordinate system.
As the operators $F_k$ are expressed in terms of the $L_\alpha\,$, the operator
$\sum_{k=1}^\infty (-1)^kF_k$ appearing on the RHS of formula
\eqref{expansion for a in powers of varphi eta rewritten}
also maps scalar functions to scalar functions.

Using formulae
\eqref{relation_Fk+1_Fk temp 1}
and
\eqref{operator_F0},
\eqref{operator_Fk},
we can rewrite
\eqref{expansion for a in powers of varphi eta rewritten}
as
\begin{equation}
\label{referee 3}
\begin{split}
a
&\simeq
a^*
-
\sum_{\gamma=1}^d \varphi_{\eta_\gamma} \sum_{k=1}^\infty \dfrac{(-1)^k}{k}\, F_{k-1}\,L_\gamma\, a
\\
&=
a^*
-
\sum_{\gamma=1}^d \varphi_{\eta_\gamma} \sum_{k=1}^\infty \dfrac{(-1)^k}{k}
\sum_{|{\boldsymbol{\alpha}}|=k-1} \dfrac{(\varphi_\eta)^{\boldsymbol{\alpha}}}{{\boldsymbol{\alpha}}!} \,L_{\boldsymbol{\alpha}}\,L_\gamma\, a
\\
&=
a^*
+
\sum_{\gamma=1}^d \varphi_{\eta_\gamma} \sum_{|{\boldsymbol{\alpha}}|\ge 0}
\dfrac{(-\varphi_\eta)^{\boldsymbol{\alpha}}}{{\boldsymbol{\alpha}}!\,(|{\boldsymbol{\alpha}}|+1)}
\,L_{\boldsymbol{\alpha}}\,L_\gamma\, a\,,
\end{split}
\end{equation}
where $a^*=a|_{x=x^*}$.
Thus, we have represented our amplitude in the form \eqref{first step} with
\begin{equation}
\label{formula for b tilde}
\tilde b_\gamma
\simeq
\sum_{|{\boldsymbol{\alpha}}|\ge 0}
\dfrac{(-\varphi_\eta)^{\boldsymbol{\alpha}}}{{\boldsymbol{\alpha}}!\,(|{\boldsymbol{\alpha}}|+1)}
\,L_{\boldsymbol{\alpha}}\,L_\gamma\, a\,.
\end{equation}

Combining \eqref{oscillatory_w_amplitude}
with \eqref{first step}
and  \eqref{formula for b tilde} and by using the identity
$$
\varphi_{\eta_\gamma}\er^{\ir\,\varphi}= \dfrac{1}{\ir} \dfrac{\partial }{\partial \eta_\gamma}\, \er^{\ir\varphi}
$$
we get, upon integration by parts,
\begin{equation*}
\mathcal{I}_\varphi(a)=\int_{T^*_yM} \er^{\ir\varphi} \, \left[ a^* +\ir \, w^{-1}  \frac{\partial}{\partial \eta_\gamma}\, \left( w   \sum_{|{\boldsymbol{\alpha}}|\ge 0} \dfrac{(-\varphi_\eta)^{\boldsymbol{\alpha}}}{{\boldsymbol{\alpha}}!\,(|{\boldsymbol{\alpha}}|+1)}\, L_{\boldsymbol{\alpha}}\right) L_\gamma \, a \right] w \, \dbar\eta\,.
\end{equation*}
Note that $a^*$ no longer depends on $x$ and the second contribution to the amplitude is now of order $m-1$. Recursive repetition of this procedure yields \eqref{mathfrak_s}--\eqref{mathfrak_sk}. The cut-off on the possible values of $|\boldsymbol{\alpha}|$ in \eqref{mathfrak_sk} follows from incorporating the information that $\varphi_\eta|_{x=x^*}=0$.
\end{proof}

\section{Weyl coefficients}
\label{Weyl coefficients}

 Let
\begin{equation*}
N(y;\lambda):= \sum_{\lambda_k<\lambda} |v_k(y)|^2
\end{equation*}
be the local counting function. When integrated over the manifold, $N(y;\lambda)$ turns into the usual (global) counting function
\begin{equation*}
N(\lambda):= \sum_{\lambda_k<\lambda} 1=\int_M N(y;\lambda)\,\rho(y)\,\dr y\,.
\end{equation*}

Let $\hat\mu:\mathbb{R}\to\mathbb{C}$ be a smooth function such that
$\hat\mu(t)=1$ in some neighbourhood of the origin
and the support of $\hat\mu$ is sufficiently small.
Here `sufficiently small' means that
$\operatorname{supp}\hat\mu\subset(-{T}_0,{T}_0)$,
where ${T}_0$ is the infimum of the lengths of all possible loops.
A loop is defined as follows.
Suppose that we have a Hamiltonian trajectory
$(x(t;y,\eta),\xi(t;y,\eta))$
and a real number $T>0$ such that
$x(T;y,\eta)=y$. We say in this case
that we have a loop of length $T$ originating
from the point $y\in M$.

We denote by
\[
\mathcal{F}[f](t)=\hat f(t)=\int_{-\infty}^{+\infty} \er^{-\ir t\lambda}\,f(\lambda)\,\dr\lambda
\]
the one-dimensional Fourier transform and by
\[
\mathcal{F}^{-1}[\hat f](\lambda)=f(\lambda)=\frac1{2\pi}\int_{-\infty}^{+\infty} \er^{\ir t\lambda}\,\hat f(t)\,\dr t
\]
its inverse.
Accordingly, we denote
$\mu:=\mathcal{F}^{-1}[\hat\mu]$.

Further on we will deal with the mollified counting function
$(N\,*\,\mu)(y,\lambda)$ rather than the original discontinuous
counting function $N(y,\lambda)$.
Here the star stands for convolution in the variable $\lambda$.
More specifically, we will deal with the derivative,
in the variable $\lambda$, of the mollified counting function.
The derivative will be indicated by a prime.

It is known
\cite{AFV,CDV,DuGu,Ivr80,Ivr84,Ivr98,SaVa}
that the function $(N'*\mu)(y,\lambda)$
admits an asymptotic expansion in integer powers of $\lambda\,$ as $\lambda\to+\infty$:
\begin{equation}
\label{expansion for mollified derivative of counting function}
(N'*\mu)(y,\lambda)=
c_{d-1}(y)\,\lambda^{d-1}
+
c_{d-2}(y)\,\lambda^{d-2}
+
c_{d-3}(y)\,\lambda^{d-3}
+\dots.
\end{equation}

\begin{definition}
\label{definition of Weyl coefficients}
We call the coefficients
$c_k(y)$ appearing in formula
\eqref{expansion for mollified derivative of counting function}
\emph{local Weyl coefficients}.
\end{definition}

Note that our definition of Weyl coefficients does not
depend on the choice of mollifier $\mu$.

Integrating \eqref{expansion for mollified derivative of counting function} in $\lambda$
and using the fact that $(N'*\mu)(y,\lambda)$ decays faster than any power of $\lambda$ as $\lambda\to-\infty$, we get
\begin{multline}
\label{expansion for mollified counting function}
(N*\mu)(y,\lambda)=
\frac{c_{d-1}(y)}{d}\,\lambda^{d}
\,+\,
\frac{c_{d-2}(y)}{d-1}\,\lambda^{d-1}
\,+\,
\dots
\,+\,
c_0(y)\,\lambda
\,+\,
\\
c_{-1}(y)\,\ln\lambda
\,+\,
b
\,-\,
c_{-2}(y)\,\lambda^{-1}
\,-\,
\frac{c_{-3}(y)}{2}\,\lambda^{-2}
\,-\,
\dots
\quad
\text{as}
\quad
\lambda\to+\infty,
\end{multline}
where $b$ is some constant.
Our Definition~\ref{definition of Weyl coefficients} is somewhat non-standard
as it is customary to call the coefficients
\[
\frac{c_{d-1}(y)}{d}\,,
\,\frac{c_{d-2}(y)}{d-1}\,,
\,\dots
\]
appearing in the asymptotic expansion
\eqref{expansion for mollified counting function} Weyl coefficients
rather than those in the asymptotic expansion
\eqref{expansion for mollified derivative of counting function}.
However, for the purposes of this paper we will stick with
Definition~\ref{definition of Weyl coefficients}.
This is the definition that was used in \cite{ASV}.

A separate question is whether one can get rid of the mollifier in \eqref{expansion for mollified counting function}.
It is known
\cite{Sa89,SaVa}
that under appropriate geometric conditions on loops we do indeed have
\begin{equation*}
\label{expansion for counting function}
N(y,\lambda)
=
\frac{c_{d-1}(y)}{d}\lambda^d
+
\frac{c_{d-2}(y)}{d-1}\lambda^{d-1}
+
o(\lambda^{d-1})
\quad
\text{as}
\quad
\lambda\to+\infty.
\end{equation*}
We do not discuss unmollified spectral asymptotics in the current paper.

The aim of this appendix is show that the small time expansion for the $g$-subprincipal symbol
of the propagator (Theorem~\ref{theorem small time}) allows us to recover in a straightforward manner the first three Weyl coefficients 
---
see also \cite{levitan}, \cite{avakumovic} and \cite{Ho68}
---
and that our result agrees with those obtained by the heat kernel method.

We have
\begin{equation}
\label{hyperbolic approach equation 1}
(N'*\mu)(y,\lambda)
\,=\,
\mathcal{F}^{-1}
\left[
\mathcal{F}
\left[
(N'*\mu)
\right]
\right](y,\lambda)
\,=\,
\mathcal{F}^{-1}
\left[
u(t,y,y)\,\hat\mu(t)
\right]\,,
\end{equation}
where $u$ is the Schwartz kernel \eqref{propagator}
of the propagator \eqref{definition of propagator}.
At  each point of the manifold
the quantity $u(t,y,y)$
is a distribution in the variable $t$ and
the construction presented in the main text of the papers allows us to write down this distribution
explicitly, modulo a smooth function.
Hence, formula \eqref{hyperbolic approach equation 1}
opens the way to the calculation of Weyl coefficients.

\begin{theorem}
\label{first three Weyl coefficients theorem}
The first three Weyl coefficients are
\begin{align}
\label{first three Weyl coefficients theorem formula 1}
&c_{d-1}(y)=\frac{S_{d-1}}{(2\pi)^d}\,,\\
\label{first three Weyl coefficients theorem formula 2}
&c_{d-2}(y)=0\,,\\
\label{first three Weyl coefficients theorem formula 3}
&c_{d-3}(y)=\frac{d-2}{12}\,\mathcal{R}(y)\,c_{d-1}(y)\,,
\end{align}
where
\begin{equation}
\label{first three Weyl coefficients theorem formula 4}
S_{d-1}=\dfrac
{2\pi^{d/2}}
{\Gamma(\frac{d}{2})}
\end{equation}
is the Riemannian volume of the $(d-1)$-dimensional unit sphere, $\mathcal{R}$ is scalar curvature and $\Gamma$ is the gamma function.
\end{theorem}

\begin{proof}
Our task is to substitute
\eqref{wave kernel = lagrangian + smoothing},
\eqref{main oscillatory integral}
into
$\,\mathcal{F}^{-1}
\left[
u(t,y,y)\,\hat\mu(t)
\right]\,$
and expand the resulting quantity in powers of $\lambda$ as $\lambda\to+\infty$.
The smooth term $\mathcal{K}$ from
\eqref{wave kernel = lagrangian + smoothing}
does not affect the asymptotic expansion,
so the problem reduces to the analysis of an explicit integral in $d+1$ variables depending on the parameter $\lambda$.
In what follows we fix a point on the manifold and drop the $y$ in our intermediate calculations.
As in the proof of Theorem~\ref{symbol identity Levi-Civita epsilon not 0},
we work in geodesic normal coordinates centred at our chosen point.

The construction presented in the main text of the paper
tells us that the only singularity of the distribution $\,u(t,y,y)\,\hat\mu(t)\,$ is at $t=0$.
Hence, in what follows, we can assume that the support of $\hat\mu$ is arbitrarily small.
In particular, this allows us to use the real-valued ($\epsilon=0$) Levi-Civita phase function.

We have
\begin{equation}
\label{first three Weyl coefficients theorem proof equation 1}
\mathfrak{a}_0(t,\eta)=1
\end{equation}
and, by Theorem~\ref{theorem small time},
\begin{equation}
\label{first three Weyl coefficients theorem proof equation 2}
\mathfrak{a}_{-1}(t,\eta)=\dfrac{\ir}{12\,\|\eta\|}\,\mathcal{R}\,t+ O(t^2)\,.
\end{equation}
The lower order terms $\mathfrak{a}_{-2},\mathfrak{a}_{-3},\ldots$ in the expansion
\eqref{homogeneous asymptotic expansion of symbol}
do not affect the first three Weyl coefficients and neither does the remainder term in
\eqref{first three Weyl coefficients theorem proof equation 2},
so further on we assume that the full symbol of the propagator reads
\begin{equation}
\label{first three Weyl coefficients theorem proof equation 3}
\mathfrak{a}(t,\eta)=1+\dfrac{\ir}{12\,\|\eta\|}\,\mathcal{R}\,t\,.
\end{equation}

Using formula
\eqref{phase function small time}
with $x=y$ we get
\begin{equation}
\label{first three Weyl coefficients theorem proof equation 4}
\varphi(t,\eta)=-\|\eta\|\,t+O(t^4)\,.
\end{equation}
Replacing $\er^{\ir\,\varphi(t,\eta)}$ by $\er^{-\ir\,\|\eta\|\,t}$
in the oscillatory integral
\eqref{main oscillatory integral}
does not affect the first three Weyl coefficients: this fact is established by
using  \eqref{first three Weyl coefficients theorem proof equation 4} and expanding
$\er^{O(t^4)}$ into a power series, with account of the fact that this $O$-term is positively homogeneous in $\eta$ of degree one
(a similar argument was used in the proofs of Theorems
\ref{theorem subprincipal identity FIO}
and
\ref{symbol identity Levi-Civita epsilon not 0}).
Hence, further on we assume that
\begin{equation}
\label{first three Weyl coefficients theorem proof equation 5}
\er^{\ir\varphi(t,\eta)}=\er^{-\ir\|\eta\|t}\,.
\end{equation}

Using formula
\eqref{b1 small time step 2}
with $x=y$ we get
\begin{equation}
\label{first three Weyl coefficients theorem proof equation 6}
\varphi_{x^\alpha\eta_\beta}(t,\eta)=\delta_\alpha{}^\beta+O(t^3)\,.
\end{equation}
Substitution of \eqref{first three Weyl coefficients theorem proof equation 6}
into
\eqref{weight w}
gives us
\begin{equation}
\label{first three Weyl coefficients theorem proof equation 7}
w(t,\eta)=1+O(t^3)\,.
\end{equation}
The remainder term in \eqref{first three Weyl coefficients theorem proof equation 7}
does not affect the first three Weyl coefficients,
so further on we assume that
\begin{equation}
\label{first three Weyl coefficients theorem proof equation 8}
w(t,\eta)=1\,.
\end{equation}

Substituting
\eqref{first three Weyl coefficients theorem proof equation 3},
\eqref{first three Weyl coefficients theorem proof equation 5}
and
\eqref{first three Weyl coefficients theorem proof equation 8}
into
\eqref{main oscillatory integral},
we conclude that formula
\eqref{hyperbolic approach equation 1}
can now be rewritten as
\begin{equation}
\label{first three Weyl coefficients theorem proof equation 9}
\begin{split}
(N'*\mu)(y,\lambda)
&=
\frac1{2\pi}\int_{\mathbb{R}^{d+1}}
\left(
1+\dfrac{\ir}{12\,\|\eta\|}\,\mathcal{R}\,t
\right)
\er^{\ir(\lambda-\|\eta\|)t}\,\hat\mu(t)\,\chi(\|\eta\|)\,\dbar\eta\,\dr t\,
\\
&+\,O(\lambda^{d-4} ).
\end{split}
\end{equation}
Here $\chi\in C^\infty(\mathbb{R})$ is a cut-off such that
$\chi(r)=0$ for $r\le1/2$
and
$\chi(r)=1$ for $r\ge1$.

Switching to spherical coordinates in $\mathbb{R}^d$, we rewrite \eqref{first three Weyl coefficients theorem proof equation 9} as
\begin{multline}
\label{first three Weyl coefficients theorem proof equation 10}
(N'*\mu)(y,\lambda)
=\\
\frac{S_{d-1}}{(2\pi)^{d+1}}\int_{\mathbb{R}^{2}}
\left(
r^{d-1}+\dfrac{\ir}{12}\,\mathcal{R}\,r^{d-2}\,t
\right)
\er^{\ir(\lambda-r)t}\,\hat\mu(t)\,\chi(r)\,\dr r\,\dr t\,
+\,O(\lambda^{d-4} ).
\end{multline}
Here $r$ is the radial coordinate and the extra factor $(2\pi)^d$ in the denominator came from \eqref{definition of dbar}.

Observe that
\[
t\,\er^{\ir(\lambda-r)t}=\ir\frac\partial{\partial r}\er^{\ir(\lambda-r)t}\,,
\]
so integrating by parts in \eqref{first three Weyl coefficients theorem proof equation 10} we simplify this formula to read
\begin{equation}
\label{first three Weyl coefficients theorem proof equation 11}
(N'*\mu)(y,\lambda)
=
\frac{S_{d-1}}{(2\pi)^{d+1}}\int_{\mathbb{R}^{2}}
r^{d-1}\,
\er^{\ir(\lambda-r)t}\,\hat\mu(t)\,\chi(r)\,\dr r\,\dr t\,
+\,O(\lambda^{d-4} )
\end{equation}
for $d=2$ and
\begin{equation}
\label{first three Weyl coefficients theorem proof equation 12}
\begin{split}
(N'*\mu)(y,\lambda)
&=
\frac{S_{d-1}}{(2\pi)^{d+1}}\int_{\mathbb{R}^{2}}
\left(
r^{d-1}+\dfrac{d-2}{12}\,\mathcal{R}\,r^{d-3}
\right)
\er^{\ir(\lambda-r)t}\,\hat\mu(t)\,\chi(r)\,\dr r\,\dr t\,
\\
&+\,O(\lambda^{d-4} )
\end{split}
\end{equation}
for $d\ge3$.

It remains only to drop the cut-off $\chi$ in formulae
\eqref{first three Weyl coefficients theorem proof equation 11}
and
\eqref{first three Weyl coefficients theorem proof equation 12}
as this does not affect the asymptotics when $\lambda\to+\infty$
and to make use of the formula
\[
\frac1{2\pi}
\int_{\mathbb{R}^{2}}
r^m\,
\er^{\ir(\lambda-r)t}\,\hat\mu(t)\,\dr r\,\dr t
=\lambda^m,
\]
which holds for $m=0,1,2,\ldots$.

We see that formulae
\eqref{first three Weyl coefficients theorem proof equation 11}
and
\eqref{first three Weyl coefficients theorem proof equation 12}
give us
\eqref{first three Weyl coefficients theorem formula 1}--\eqref{first three Weyl coefficients theorem formula 3}.
\end{proof}

\

As a final step, let us show that Theorem~\ref{first three Weyl coefficients theorem} agrees with the classical heat kernel expansion.
To this end, let us introduce the (local) heat  trace
\begin{equation}
\label{heat kernel 1}
Z(y,t):=
\int_{-\infty}^{+\infty}\er^{-t\lambda^2}\,N'(y,\lambda)\,\dr\lambda\,
=
\int_0^{+\infty}\er^{-t\lambda^2}\,N'(y,\lambda)\,\dr\lambda\,
+
\,\frac1{\operatorname{Vol}(M,g)}\,.
\end{equation}
If we now replace $\,N'(y,\lambda)\,$ in formula \eqref{heat kernel 1} with its mollified version
$\,(N'*\mu)(y,\lambda)\,$ this gives an error, but this error can be easily estimated:
\begin{equation}
\label{heat kernel 2}
Z(y,t)=
\int_0^{+\infty}\er^{-t\lambda^2}\,(N'*\mu)(y,\lambda)\,\dr\lambda\,+\,O(1)\quad\text{as}\quad t\to 0^+.
\end{equation}

Substituting
\eqref{expansion for mollified derivative of counting function}
and
\eqref{first three Weyl coefficients theorem formula 1}--\eqref{first three Weyl coefficients theorem formula 3}
into
\eqref{heat kernel 2}, we get
\begin{equation}
\label{heat kernel 3}
Z(y,t)=
c_{d-1}(y)
\int_0^{+\infty}\er^{-t\lambda^2}\,
\lambda^{d-1}
\,\dr\lambda\,+\,O(1)\quad\text{as}\quad t\to 0^+
\end{equation}
for $d=2$,
\begin{multline}
\label{heat kernel 4}
Z(y,t)=
\int_0^{+\infty}\er^{-t\lambda^2}
\left(
c_{d-1}(y)\,\lambda^{d-1}
+
c_{d-3}(y)\,\lambda^{d-3}
\right)
\dr\lambda\,+\,O(|\ln t|)\\
\text{as}\quad t\to 0^+
\end{multline}
for $d=3$,
and
\begin{multline}
\label{heat kernel 5}
Z(y,t)=
\int_0^{+\infty}\er^{-t\lambda^2}
\left(
c_{d-1}(y)\,\lambda^{d-1}
+
c_{d-3}(y)\,\lambda^{d-3}
\right)
\dr\lambda\,+\,O\bigl(t^{(3-d)/2}\bigr)\\
\text{as}\quad t\to 0^+
\end{multline}
for $d\ge4$.

We have
\begin{equation}
\label{heat kernel 6}
\int_{0} ^{+\infty}\er^{-z^2} \,z^{d-1}\,\dr z\,=\,\dfrac{\Gamma(\frac{d}{2})}{2}\,.
\end{equation}
We also have
\begin{equation}
\label{heat kernel 7}
\int_{0} ^{+\infty}\er^{-z^2} \,z^{d-3}\,\dr z\,=\,\dfrac{\Gamma(\frac{d}{2}-1)}{2}
\,= \dfrac{\Gamma(\frac{d}{2})}{d-2}
\end{equation}
for $d\ge3$.

Using
\eqref{first three Weyl coefficients theorem formula 1}--\eqref{first three Weyl coefficients theorem formula 4},
\eqref{heat kernel 6}
and
\eqref{heat kernel 7}
we can rewrite formulae
\eqref{heat kernel 3}--\eqref{heat kernel 5}
as a single formula
\begin{equation}
\label{heat kernel 8}
Z(y,t)=
\begin{cases}
(4\pi t)^{-d/2}
+O(1)\quad\text{for}\quad d=2,
\\
(4\pi t)^{-d/2}
\left(
1+\frac16\,\mathcal{R}(y)\,t
\right)
+O(|\ln t|)\quad\text{for}\quad d=3,
\\
(4\pi t)^{-d/2}
\left(
1+\frac16\,\mathcal{R}(y)\,t
\right)
+O\bigl(t^{(3-d)/2}\bigr)\quad\text{for}\quad d\ge4
\end{cases}
\end{equation}
as $t\to 0^+$.

It is known 
\cite{minak},
\cite[Ch.~III, E.IV.]{berger},
\cite[Section~3.3]{rosenberg},
that for all $d\ge2$ the heat trace admits the expansion
\begin{equation}
\label{heat kernel 9}
Z(y,t)=
(4\pi t)^{-d/2}
\left(
1+\frac16\,\mathcal{R}(y)\,t
\right)
+O\bigl(t^{(4-d)/2}\bigr)
\quad\text{as}\quad t\to 0^+.
\end{equation}
We see that our result \eqref{heat kernel 8} agrees with the classical formula  \eqref{heat kernel 9}.

\end{appendices}

\end{document}